\documentclass[review,onefignum,onetabnum]{siamart250211}

\usepackage{amssymb}
\usepackage{amsmath}
\usepackage{mathtools}
\usepackage{xcolor}
\usepackage{subfigure}
\usepackage{hyperref}
\usepackage{cleveref}
\usepackage{graphicx}
\usepackage{subcaption} 
\usepackage{algorithm}
\usepackage{algpseudocode}
\usepackage{adjustbox}

\usepackage[colorinlistoftodos,bordercolor=orange,backgroundcolor=orange!20,linecolor=orange,textsize=scriptsize]{todonotes}

\newcommand{\bx}{\boldsymbol{x}}
\newcommand{\by}{\boldsymbol{y}}
\newcommand{\bz}{\boldsymbol{z}}

\newcommand{\Y}{\mathcal{Y}}
\newcommand{\Z}{\mathcal{Z}}

\newcommand{\bp}{\boldsymbol{p}}
\newcommand{\R}{\mathbb{R}}
\newcommand{\Rb}{\bar{\mathbb{R}}}
\newcommand{\Pp}{\mathcal{P}}
\newcommand{\F}{\mathcal{F}}
\newcommand{\Cc}{\mathcal{C}}
\newcommand{\Ll}{\mathcal{L}}

\newcommand{\Uu}{\mathcal{U}}

\DeclareMathOperator*{\BUC}{BUC}
\DeclareMathOperator*{\intr}{int}
\DeclareMathOperator*{\dom}{dom}
\DeclareMathOperator*{\argmin}{argmin}
\DeclareMathOperator*{\Jac}{Jac}
\DeclareMathOperator{\prox}{prox}
\DeclareMathOperator{\inter}{int}
\DeclareMathOperator{\dist}{dist}

\newsiamremark{remark}{Remark}
\newcommand{\newsiamassump}[2]{
  \theoremstyle{plain}
  \theoremheaderfont{\normalfont \bf}
  \theorembodyfont{\normalfont}
  \theoremseparator{}
  \theoremsymbol{}
  \newtheorem{#1}{#2}
}

\newsiamassump{assumption}{Assumption}

\title{
Algorithms and Differential Game Representations for Exploring Nonconvex Pareto Fronts in High Dimensions
\thanks{Submitted to the editors \today.} 
\funding{This research is supported by the DARPA DIAL grant HR00112490484, a Laboratory University Collaboration Initiative (LUCI) award sponsored by the Basic Research Office (BRO) of the Office of the Under Secretary of Defense for Research and Engineering (OUSD R\&E), 
the U.S. Department of Energy (DOE), Office of Science, Advanced Scientific Computing Research (ASCR) program under the Scientific Discovery through Advanced Computing (SciDAC) Institute “LEADS: LEarning-Accelerated Domain Science,” Subcontract \#831126 under DE-AC05-76RL01830, the Office of Naval Research (ONR) In-House Laboratory Independent Research Program (ILIR) at NAWCWD, and a Department of Defense (DoD) SMART Scholarship for Service SEED Innovation Award. \textit{Distribution Statement A. Approved for Public Release; Distribution is Unlimited. PR 26-0024.}} }

\author{
Shanqing Liu\thanks{Division of Applied Mathematics, Brown University, Providence, RI~\email{shanqing\_liu@brown.edu}, \email{youngkyu\_lee@brown.edu}, \email{jerome\_darbon@brown.edu}.}
\and
Paula Chen\thanks{Naval Air Warfare Center Weapons Division China Lake, China Lake, CA~\email{paula.x.chen.civ@us.navy.mil}.}
\and
Youngkyu Lee\footnotemark[2]
\and
J\'er\^ome Darbon\footnotemark[2]
}

\begin{document}
\nolinenumbers

\maketitle

\begin{abstract} 
We develop a new Hamiton-Jacobi (HJ) and differential game approach for exploring the Pareto front of (constrained) multi-objective optimization (MOO) problems. 
Given a preference function, we embed the scalarized MOO problem into the value function of a parameterized zero-sum game, whose upper value solves a first-order HJ equation that admits a Hopf-Lax representation formula.  
For each parameter value, this representation yields an inner minimizer that can be interpreted as an approximate solution to a shifted scalarization of the original MOO problem.
Under mild assumptions, the resulting family of solutions maps to a dense subset of 
the weak Pareto front. 
Finally, we propose a primal-dual algorithm based on this approach for solving the corresponding optimality system. 
Numerical experiments show that our algorithm mitigates the curse of dimensionality (scaling polynomially with the dimension of the decision and objective spaces) and is able to expose continuous curves along nonconvex Pareto fronts in 100D in just $\sim$100 seconds.
\end{abstract}



\section{Introduction}

\subsection{Motivation and context}

Many important problems in economics, game theory, engineering, and data-driven modeling involve the simultaneous optimization of multiple objectives (see, e.g., \cite{stadler1988fundamentals,marler2004survey}). For example, robotics and autonomous system development often involves balancing performance, energy consumption, and safety constraints; aerospace applications generally require tradeoffs between fuel usage,
speed, and structural loads; and machine learning-based tasks inherently necessitate weighing accuracy, robustness, and interpretability.  
All of these cases can be formulated as a \emph{multi-objective optimization} (MOO) problem of the form 
\begin{equation}\label{MOO_obj}
\min_{u \in U}\ell(u) := ( \ell_1(u), \ell_2(u), \dots, \ell_N(u) )\in\R^N.
\end{equation}
In general, one cannot identify a single $u^*$ that simultaneously minimizes all components of the vector-valued objective function $\ell$, and optimality is often defined using \emph{Pareto optimality}~\cite{pareto1964cours} (see also~\cite{yu1974cone}). Broadly speaking, a solution is \emph{Pareto optimal} if there does not exist another solution that can improve any criteria without deteriorating another criterion.
The set of all Pareto optimal decisions (the \emph{Pareto optimal set}) and its image under $\ell$ (the \emph{Pareto front}) then summarize the attainable tradeoffs and form fundamental objects of interest in MOO. 

One fundamental challenge in MOO is dealing with \emph{nonconvex} Pareto fronts (i.e., the Pareto front is the boundary of a nonconvex set).
In particular, a common approach for solving MOO problems is to use a \emph{weighted-sum} \emph{scalarization} that combines the vector-valued objective into a single scalar criterion $\sum_{i=1}^N w_i \ell_i(u)$, where the weights $w \in (\R_+)^N$ encode the relative ``importance'' of each objective~\cite{jahn2011vector}.
However, such weighted sum scalarizations can only recover the convex envelope of the Pareto front and hence cannot be used to explore nonconvex portions (see, e.g., \cite{wierzbicki1982mathematical,lee2024automatic,cao2025automatic}).
This limitation motivates the need for alternative techniques and algorithms for systematically exploring nonconvex geometries.

Another difficulty is the \emph{curse of dimensionality}, which appears both in the dimension $N$ of objectives  and in the dimension $d$ of the decision space.
Indeed, the Pareto front of an $N$-objective optimization problem forms an $(N-1)$-dimensional surface in the objective space.
Hence, obtaining accurate approximations with reasonable coverage of the true front often involves solving a large collection of parametric scalarized problems, and this effort can grow quickly with $N$ (e.g., see \cite{novak2008tractability, coello2007evolutionary}).
At the same time, each scalarized subproblem must be solved in a
$d$-dimensional \emph{search space}, which in turn generally induces the usual curse of dimensionality associated with high-dimensional optimization and dynamic programming~\cite{bellman1957dynamic,bertsekas2012dp}.
As such, practical exploration of high-dimensional Pareto fronts requires efficient computational procedures that avoid uniform gridding of the decision and/or objective space.  

\subsection{Contributions}

\begin{figure}
    \centering
    \includegraphics[width=0.99\linewidth]{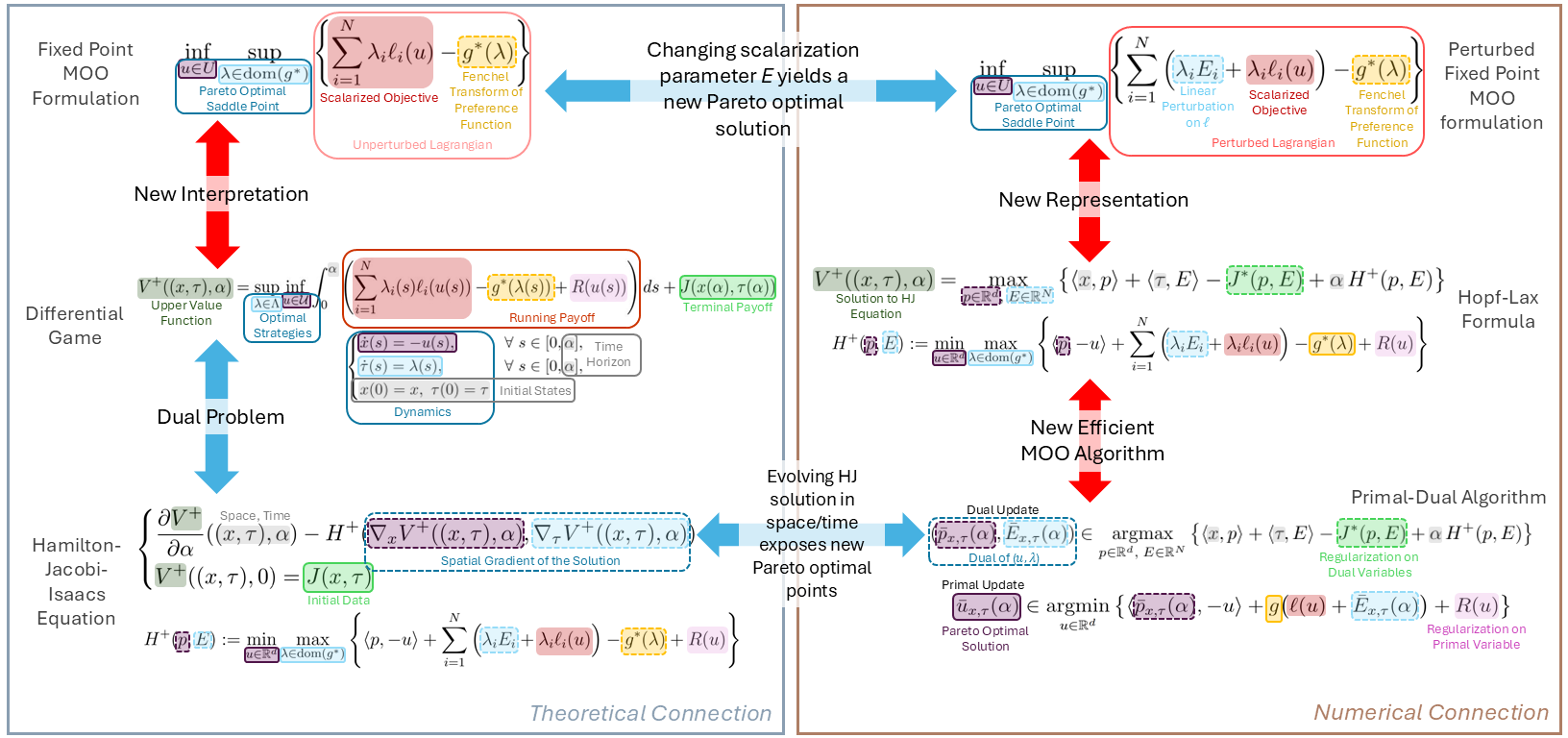}
    \caption{\textit{Overview of our approach.} Evolution of the Pareto front is characterized by an associated differential game and HJ equation. A Hopf-Lax representation then yields an efficient primal-dual scheme for the numerical exploration of high-dimensional, potentially nonconvex Pareto fronts. Colors indicate quantities that are equivalent between problem formulations (solid/no lines) and their duals (dotted lines).}
    \label{fig:overview}
\end{figure}

We introduce a new Hamilton-Jacobi (HJ) and differential game approach to exploring the Pareto front of (constrained) MOO problems. 
In particular, we provide a \textit{new Hopf-Lax representation for the Pareto front }that is interpretable in terms of HJ theory and differentiaLl games as follows. Given a smooth vector objective $\ell:\R^d\to\R^N$ and a monotone \emph{preference function} $g:\R^N\to\R$, 
minimizing $g(\ell(u)+E)$ yields Pareto-relevant points for appropriate  $E\in\R^N$.  
Rather than addressing each scalarization independently, we embed this family of scalarizations into a parameterized zero-sum game whose
upper value $V^+$ satisfies a first-order HJ equation. 
Under mild assumptions, $V^+$ admits a Hopf-Lax representation, which, for each parameter value, produces dual variables $(p_\alpha,E_\alpha)$ and 
an inner minimizer $u_\alpha$ that solves an approximate shifted scalarization $u\mapsto g(\ell(u)+E_\alpha)$ with an explicit error controlled by a Bregman divergence. Cluster points of parameterized sequences along which the regularization error vanishes then solve the original saddle point problem induced by $g$, thereby
generating weak Pareto optimal points. 

In turn, this mathematical connection suggests efficient numerical solvers based on the optimality system associated with the inner minimization problem for $u_\alpha$. 
Along these lines, we establish a primal-dual algorithm for Pareto front exploration based on this approach  that scales polynomially in both $N$ and $d$ and provide corresponding Kurdyka-Łojasiewicz (KL)-based convergence analysis. Numerical experiments illustrate that our algorithm is able to expose \textit{continuous curves} along \textit{nonconvex} Pareto fronts, while \textit{mitigating the curse of dimensionality} (with cases that scale up to $d=100$, $N=5$). 
Our approach is summarized in~\Cref{fig:overview}. 


\subsection{Paper organization}
In~\Cref{sec-prem}, we provide background on MOO, HJ equations, and differential games. 
In~\Cref{sec-main}, we establish key results rigorously justifying the representation of the Pareto front using HJ equations and differential games. A primal-dual algorithm is also proposed to numerically solve the corresponding optimality system.  
In~\Cref{sec-con}, we extend the results of the previous section to constrained MOO problems. 
Numerical examples are given in~\Cref{sec-numerics}. A brief summary and some future directions are provided in~\Cref{summary}.

\section{Preliminaries}\label{sec-prem}

\subsection{Notation}\label{sec:notation}

Let $X$ be a Banach space and $X^*$ be its topological dual space. 
We denote by $\R$ the set of real numbers, $\R_+$ the set of nonnegative real numbers, and $\Rb = \R \cup \{+\infty \}$. 
Consider a function $f:X \to \Rb$ with domain $\dom(f)$. Then, the \textit{Legendre-Fenchel transform}~\cite{bonnans2006numerical,nocedal2006numerical} $f^*: X^*\to \Rb$ of $f$ is defined by
\begin{equation}\label{legendre}
f^*(\bp) = \sup_{\bx \in X} \{ \langle \bp,\bx \rangle - f(\bx) \}, \ \forall \ \bp \in X^*,
\end{equation}
where $\langle \cdot,\cdot \rangle$ denotes the duality pairing for every $(\bp,\bx) \in X^* \times X$. 

Let $x,y \in \R^d$. We say that $x \leq y$ if $x_i \leq y_i$ for every $i \in \{1,2,\dots,d \}$ and $x < y$ if, in addition, there exists at least one $j\in \{1,2,\dots,d \}$ such that $x_j < y_j$. 
We call a function $g :\R^d \to \R$ \textit{nondecreasing} if and only if for every $x,y \in \R^d$ with $x \leq y$, one has $g(x)\leq g(y)$. Similarly, we call a function $g :\R^d \to \R$ \textit{increasing} if and only if for every $x,y \in \R^d$ with $x < y$, one has $g(x)< g(y)$. Note that these definitions are special cases of monotonicity with respect to a cone. In general, let $K \subset \R^d$ be a convex cone and let $\intr K$ denote its interior. Then, the cone $K$ induces a partial ordering defined by $ x \preceq_Ky$ if and only if $y-x \in K$. A function $g:\R^d \to \R$ is called \textit{$K-$nondecreasing} if $x\preceq_Ky$ implies $g(x) \leq g(y)$ and \textit{strictly $K-$increasing} if $x\prec_Ky$ (that is, $y-x \in \intr K$) implies $g(x) < g(y)$. In the case where $K = \R_+^d$, this partial ordering coincides with the coordinatewise order defined above. 

For any Banach spaces $X$ and $Y$, we denote by $\BUC(X;Y)$ the space of bounded, uniformly continuous functions from $X$ to $Y$. We use $\Cc^1(X;Y)$ to denote the space of continuously differentiable (in the Fr\'echet sense) functions from $X$ to $Y$. 
For $T>0$ and $d\geq 1$, we denote by $W^{1,1}([0,T];\R^d)$ the Sobolev space of functions from $[0,T]$ to $\R^d$ that are integrable and whose weak derivatives are also integrable in $[0,T]$.   

\subsection{Multi-objective optimization and Pareto optimality}

We are interested in solving multi-objective optimization (MOO) problems. The goal of a MOO problem is to minimize an $N-$dimensional vector-valued function $\ell : U \subseteq \R^d \to \R^N$, as in~\eqref{MOO_obj}.  
We make the following basic assumptions on $\ell$. 
\begin{assumption}\label{assump_f} \textit{}
    \begin{enumerate}
        \item[i.] $\{\ell_i \}_{i \in \{1,2,\dots, N \}}$ are lower semicontinuous (lsc) on $U$. 
        \item[ii.] $\{\ell_i \}_{i \in \{1,2,\dots, N \}}$ are \textit{proper}, that is $\ell_i(u)>-\infty$ for every $u \in U$ and there exists a $u^0\in U$ such that $\ell_i(u^0) < +\infty$. 
    \end{enumerate}
\end{assumption}

In general, the MOO problem~\eqref{MOO_obj} does not admit a classical solution that minimizes all of the objective functions $\ell_i$ simultaneously. Therefore, the notion of an optimal solution must be understood in terms of a suitable concept of optimality. In this work, we focus on identifying candidates within the framework of Pareto optimality~\cite{censor1977pareto,miettinen1999nonlinear}. To do so, we recall the definition of dominance in the sense of minimizing $\ell$, which is closely related to Pareto optimality.

\begin{definition}\label{MOO_dominance} For every $u^1, u^2 \in U$, 
    \begin{enumerate}
        \item[i.] $u^1$ \textit{dominates} $u^2$ (denoted by $u^1 \succeq u^2$) if and only if $\ell(u^1) < \ell(u^2)$. 
        \item[ii.] $u^1$ \textit{strictly dominates} $u^2$ (denoted by $u^1 \succ u^2$) if and only if $\ell_i(u^1) < \ell_i(u^2),  \forall  i$. 
    \end{enumerate}
\end{definition}

\begin{definition}\label{Pareto}
For any $u\in U$,
\begin{enumerate}
    \item[i.] $u$ is a \textit{strong Pareto optimal solution} if and only if there is no solution that dominates $u$, i.e., there is no other $u'$ that improves one objective without worsening another objective. We denote by $\Pp(\ell)$ the set of all strong Pareto optimal solutions of~\eqref{MOO_obj}. 
    \item[ii.] $u$ is a \textit{weak Pareto optimal solution} if and only if there is no solution that strictly dominates $u$, i.e., there is no other $u'$ that can improve all objectives at once. We denote by $\Pp_w(\ell)$ the set of weak Pareto optimal solutions of~\eqref{MOO_obj}.
\end{enumerate}
\end{definition}

It follows directly from~\Cref{Pareto} that $\Pp(\ell) \subseteq \Pp_w(\ell)$. 
\begin{definition}\label{MOO_ParetoFront} 
In the context of minimizing $\ell$, we call the \textit{Pareto front} $\F(\ell)$ (\textit{weak Pareto front} $\F_w(\ell)$, resp.) the image of the set of all strong Pareto optimal solutions $\Pp(\ell)$ (the set of all weak Pareto optimal solutions $\Pp_w(\ell)$, resp.); that is
    \begin{equation}\label{pareto_front}
    \F(\ell) := \{\ell(u) \mid u \in \Pp(\ell) \},  \quad \F_w(\ell) := \{\ell(u) \mid u \in \Pp_w(\ell) \} .
    \end{equation}
\end{definition}

Note that when there exists a solution $u^*$ that optimizes all of the objective functions $\{\ell_i\}_{i\in\{1,2,\dots, N\}}$ at once, then the Pareto front is a singleton
\begin{equation}
\F(\ell) = \F_w(\ell) = \left\{ ( \ell_1(u^*), \ell_2(u^*), \dots, \ell_N(u^*) ) \right\}.
\end{equation}


\subsection{Hamilton-Jacobi equations and differential games}
In this work, we reinterpret exploration of the Pareto front in MOO as the evolution of a differential game and its associated Hamilton-Jacobi (HJ) equation. To do so, we consider an HJ equation \eqref{HJ_general}  of the form
\begin{equation}
\label{HJ_general}
    \begin{dcases}
        F(x,\nabla_x V) := \frac{\partial V}{\partial t}(x,t) - H(x,\nabla_x V(x,t)) = 0, & (x,t) \in \R^d \times \R_+, \\
        V(x,0) = J(x), & x\in \R^d,
    \end{dcases}
\end{equation}
where $\nabla_x V$ denotes the gradient of $v$ with respect to $x$. We assume $H:\R^d \times \R^d \to \R$ is continuous and $J \in \BUC(\R^d;\R)$.  
Then, the viscosity solution of this equation is closely related to the dynamic programming principle for optimal control and two-player zero-sum differential games. 
\begin{definition}[see~\cite{crandall1983viscosity,crandall1984some}]
    Let $V\in \BUC(\R^d \times[0,T]; \R)$. 
    \begin{enumerate}
    \begin{subequations}
        \item $V$ is a \textit{viscosity subsolution} of~\eqref{HJ_general} if $V(\cdot, 0)\leq J$ and if for every test function $\phi \in \Cc^1(\R^d \times [0,T]; \R)$ and for all local maximizers $(x_0,t_0) \in \R^d \times [0,T]$ of the function $V-\phi$, we have 
        $\frac{\partial \phi}{\partial t}(x_0,t_0) - H(x_0,\nabla \phi (x_0,t_0)) \leq 0.$ 
        \item $V$ is a \textit{viscosity supersolution} of~\eqref{HJ_general} if $V(\cdot, 0)\geq J$ and if for every test function $\phi \in \Cc^1(\R^d \times [0,T]; \R)$ and for all local minimizers $(x_0,t_0) \in \R^d \times [0,T]$ of the function $V-\phi$, we have 
        $\frac{\partial \phi}{\partial t}(x_0,t_0) - H(x_0,\nabla \phi (x_0,t_0)) \geq 0.$  
        \end{subequations}
        \item $V$ is a \textit{viscosity solution} of~\eqref{HJ_general} if it is both a viscosity subsolution and viscosity supersolution of~\eqref{HJ_general}.
    \end{enumerate}
\end{definition}

We briefly recall the well-established relationship between the HJ equation~\eqref{HJ_general} and differential games (e.g., see \cite{elliott1972existence,evans1984differential}). 
For every $0<t\leq T$, consider a controlled dynamical system of the form
\begin{equation}\label{dynamcs_game}
    \begin{aligned}
        & \dot{\bx}(s) = f(\bx(s), \by(s), \bz(s)), \ s \in [0,t] \\ 
    \end{aligned}
\end{equation}
with initial condition $x(0) = x$, where $\by \in \Y:= \{ \by :[0,t] \to Y  \mid \text{ $\by(\cdot)$ is measurable}  \}$ and $\bz \in \Z:=  \{ \bz :[0,t] \to Z  \mid \text{ $\bz(\cdot)$ is measurable} \}$ are the \textit{strategies} of the two players of the game. Under standard regularity assumptions on $f : \R^d \times Y \times Z \to \R^d$ (i.e., $f$ is bounded, Lipschitz in $\bx$, and uniformly continuous in $\by$ and $\bz$), it follows that, for any given initial condition $x \in \R^d$ and strategies $\by \in \Y$, $\bz \in Z$, the system \eqref{dynamcs_game} admits a unique solution $\bx_x^{\by,\bz} \in W^{1,1}([0,t];\R^d)$. 
Now consider a \textit{payoff} functional
\begin{equation}\label{pay_game}
    C(x,t; \by(\cdot),\bz(\cdot) ) = \int_{0}^t h(\bx_x^{\by,\bz} (s), \by(s),\bz(s)) ds + J(\bx_x^{\by,\bz}(t)) \ 
\end{equation}
associated with the dynamical system~\eqref{dynamcs_game}. To play the game, Player 1 seeks to maximize~\eqref{pay_game} with strategy $\by\in\Y$, while Player 2 seeks to minimize~\eqref{pay_game} with strategy $\bz \in \Z$. Here, we also make standard regularity assumptions on the running payoff $h:\R^d \times Y \times Z\to \R$; that is, $h$ is bounded, Lipschitz continuous
in $\bx$, and uniformly continuous in $\by$ and $\bz$. 

The dynamic programming principle for differential games leads to the notions of \textit{upper} and \textit{lower value functions}. 
To derive these value functions rigorously, one often introduces \emph{nonanticipative strategies}~\cite{elliott1972existence, pierre2010introduction} (for brevity, we omit the formal definition here). 
The upper value function $V^+:\R^d \times[0,T] \to \R$ and the lower value function $V^-:\R^d\times[0,T] \to \R$ are defined as follows:
\begin{equation}
    V^+(x,t) := \sup_{\alpha} \inf_{\bz \in \Z}  \  C(x,t;\alpha[\bz],\bz),
    \qquad
    V^-(x,t):=\inf_{\beta} \sup_{\by \in \Y}\  C(x,t;\by,\beta[\by]).
\end{equation}

\begin{proposition}[see, for instance, \cite{evans1984differential}] \textit{ }
    \begin{enumerate}
        \item $V^+$ is the viscosity solution of the HJ equation~\eqref{HJ_general} with Hamiltonian $H=H^+$, where
        $\label{upper_Hamiltonian}
            H^+ (x,p) = \min_{z \in Z} \max_{y \in Y} \{ f(x,y,z) \cdot p +c(x,y,z) \}.
            $
        \item $V^-$ is the viscosity solution of the HJ equation~\eqref{HJ_general} with Hamiltonian $H=H^-$, where
         $   H^- (x,p) = \max_{y \in Y}\min_{z \in Z}  \{  f(x,y,z) \cdot p + c(x,y,z) \}. $ 
    \end{enumerate}
\end{proposition}
We call $H^+$ the \textit{upper Hamiltonian} and $H^-$ the \textit{lower Hamiltonian}. 
HJ equations of this type are often referred to as Hamilton--Jacobi--Isaacs (HJI) equations~\cite{isaacs1999differential}.

\section{Representation and exploration of the Pareto front via differential games and HJ equations}\label{sec-main}

In this section, we study the MOO problem~\eqref{MOO_obj} with $U = \R^d$. 
In particular, we provide a differential game formulation and HJ representation formula that correspond to the evolution of the Pareto front.

\subsection{A fixed point formulation for identifying one Pareto optimal solution}\label{sec:fixedpt}
One way to identify a Pareto optimal solution is to apply a \textit{preference function} $g: \R^N \to \R$ that transforms the MOO problem into a scalarized problem of the form
\begin{equation}\label{scala}
    \inf_{u \in \R^d} g\circ \ell(u).
\end{equation}
We make the following assumptions on $g$.

\begin{assumption}\label{assump_g} \textit{ 
}
\begin{enumerate}
\item[i.] $g$ is convex, proper, and lsc on $\R^N$. 
    \item[ii.] $g$ is strictly $(\R_+)^N$ increasing in the cone-interior sense.  
\end{enumerate}
\end{assumption}
    Note that under~\Cref{assump_g}, we have $\dom(g^*) \subseteq (\R_+)^N$.

Given a fixed set of weights $\{ \lambda_i\}_{i \in \{1,2,\dots,N \}}$, where $\lambda_i >0$ for every $i$ and $\sum_i \lambda_i = 1$, 
two commonly used choices of scalarization functions $g$ that satisfy~\Cref{assump_g} are weighted sum approximations  
    $g^1(\ell):= \sum_{i=1}^N \lambda_i \ell_i$ 
and Chebyshev (or weighted max) approximations  
 $    g^\infty(\ell): = \max  \{\lambda_1 \ell_1, \lambda_2 \ell_2 , \dots,\lambda_N\ell_N \}.$ 

\begin{proposition}\label{prop-onePareto}
     Under~\Cref{assump_f} and~\Cref{assump_g}, a weak Pareto optimal solution of problem~\eqref{MOO_obj} can be obtained by solving the minimax problem: 
     \begin{equation}\label{saddle-one}
         \inf_{u \in U} \sup_{\lambda \in \dom(g^*)} \left\{ \sum_{i=1}^N \lambda_i \ell_i(u) -g^*(\lambda)\right\}.
     \end{equation}
\end{proposition}
\begin{proof}
    Since $g$ is proper, lsc, and convex, the Fenchel-Moreau theorem (e.g., see \cite{rockafellar1970convex, hiriart1993convex}) gives $g = g^{**}$ , i.e.,
       $ g(z) = \sup_{\lambda \in \R^N} \{ \langle \lambda, z \rangle - g^*(\lambda) \} , \ \forall \ z \in \R^N. $ 
    Taking now $z = \ell(u)$, we have 
     $    g(\ell(u)) = \sup_{\lambda} \{ \langle \lambda, \ell(u) \rangle - g^*(\lambda) \}. $
    Taking the infimum over $u$ on both sides then yields exactly the min-max formulation in~\eqref{saddle-one}. 

    Now, let $u^*$ be a minimizer of $g(\ell(u))$. Suppose, by contradiction, that $u^* \notin \Pp_w(\ell)$. Then there exists a $u$ such that $\ell_i(u) < \ell_i(u^*)$ for all $i$, that is $\ell(u^*) - \ell(u) \in \inter (\R_+)^N $. 
    However, by strict $(\R_+)^N-$monotonicity on $\inter (\R_+)^N$, we have that $g(\ell(u)) <g(\ell(u^*))$, which contradicts that $u^*$ minimizes $g\circ\ell$. Hence, $u^* \in \Pp_w(\ell)$. 
\end{proof}

\subsection{Pareto front exploration via a differential game}
\Cref{prop-onePareto} gives a saddle point formulation for \textit{one} Pareto optimal solution. 
 We now propose to expose \textit{continuous curves} along the Pareto front by using the solution of a differential game to continuously evolve Pareto optimal saddle points.  

Let $0 \leq \alpha $. 
Denote 
   $      \Uu:= \{u:[0,\alpha] \to \R^d \mid u (\cdot) \text{ is measurable} \} $ and $  \Lambda := \{\lambda:[0,\alpha] \to (\R_+)^N \mid \lambda (\cdot) \text{ is measurable}\}$, and  
let $x \in \R^d$ and $\tau \in \R^N$. Consider the following controlled dynamical system:
    \begin{equation}\label{dynamics}
        \left\{ 
        \begin{aligned}
            & \dot{x}(s) = -u(s), \ \forall \ s\in[0, \alpha], \\
            & \dot{\tau}(s) = \lambda(s), \ \forall \ s\in[0, \alpha]
        \end{aligned}
        \right.
    \end{equation}
    with initial condition $x(0) = x, \ \tau(0) = \tau$ 
    and $u \in \Uu, \lambda \in \Lambda$. 
We consider a differential game associated with the dynamical system~\eqref{dynamics}. The two players in the game compete using strategies $u$ and $\lambda$ to minimize and maximize, respectively, the cost  
\begin{equation}
\begin{adjustbox}{width=\textwidth}
    $C((x,\tau),\alpha; u(\cdot),\lambda(\cdot))  := \int_{0}^{\alpha} \Big( \sum_{i=1}^N \lambda_i(s) \ell_i(u(s)) -g^*(\lambda(s)) +R(u(s)) \Big) ds + J(x(\alpha),\tau(\alpha)),$
\end{adjustbox}
\end{equation} 
where $J$ is a terminal cost and $R:\R^d\to\Rb$ acts as a regularizer. 
Moreover, we make the following assumptions on $J$ and $R$. 
\begin{assumption}\label{assume_JR} \textit{ }
\begin{enumerate}
    \item $J:\R^{d}\times \R^N \to \Rb$ is  proper, lsc, and convex. 
    \item $R:\R^d\to\Rb $ is proper, lsc,  $\Cc^1$, and \textit{$\mu$–strongly convex} (i.e., $u\mapsto R(u)-\tfrac{\mu}{2}\|u\|^2$ is convex for some $\mu>0$) and has full domain $\dom (R)=\R^d$.
\end{enumerate}
\end{assumption}

For every initial condition $(x,\tau) \in \R^d \times (\R_+)^N$, we consider two value functions, the \textit{upper value function} $V^+:(\R^d \times (\R_+)^N) \times \R_+ \to \R$ and the \textit{lower value function} $V^-:(\R^d \times (\R_+)^N) \times \R_+ \to \R$, defined as follows:
    \begin{equation}\label{upper_value}
        V^+((x,\tau),\alpha) : =\sup_{\lambda \in \Lambda} \inf_{u \in \Uu}  \ C((x,\tau),\alpha; u,\lambda),\ V^-((x,\tau),\alpha) : = \inf_{u \in \Uu} \sup_{\lambda \in \Lambda} \ C((x,\tau),\alpha; u,\lambda).
    \end{equation}
Moreover, for every $(p,E) \in \R^d \times \R^N$, let us consider two Hamiltonians of the form
    \begin{equation}\label{Hamiltonians}
    \begin{aligned}
       & H^+(p,E):= \min_{u \in \R^d} \max_{\lambda \in \dom(g^*)} \Big\{ \langle p, -u \rangle + \sum_{i=1}^N \Big( \lambda_i E_i + \lambda_i \ell_i(u) \Big) - g^*(\lambda) + R(u)  \Big\}, \\ 
       & H^-(p,E):= \max_{\lambda \in \dom(g^*)} \min_{u \in \R^d}  \Big\{ \langle p, -u \rangle + \sum_{i=1}^N \Big( \lambda_i E_i + \lambda_i \ell_i(u) \Big) - g^*(\lambda) + R(u)  \Big\}.
        \end{aligned}
    \end{equation}
\begin{proposition}
$V^+$ and $V^-$ are the viscosity solutions of the HJI equations 
\begin{equation}
    F(x, \nabla_{x,\tau}V^+((x,\tau),\alpha)) = 0, \quad F(x, \nabla_{x,\tau}V^-((x,\tau),\alpha) = 0,
\end{equation} 
where $F$ takes the form of~\eqref{HJ_general} with Hamiltonians $H^+$ and $H^-$ in~\eqref{Hamiltonians}, respectively, and with initial data given by the terminal cost $J$. 
\end{proposition}

\subsubsection{The upper Hamiltonian}

We focus on the upper value function $V^+$ and its associated Hamiltonian $H^+$ to draw our connection between differential games and MOO. 
By elementary computations, we have that
\begin{equation}\label{H+_1}
\begin{aligned}
    H^+(p,E) 
    &= \min_{u \in \R^d} \max_{\lambda \in \dom(g^*)} \Big\{ \langle p, -u \rangle + \sum_{i=1}^N \Big( \lambda_i E_i + \lambda_i \ell_i(u) \Big) - g^*(\lambda) +R(u)  \Big\} \\ 
    &= \min_{u\in \R^d} \{\langle p, -u \rangle + g \circ(\ell(u) + E) + R(u)\} = -(g \circ (\ell(\cdot) + E) + R)^*(p).
    \end{aligned}
\end{equation}
By rewriting $H^+$ in this way, it is clear that $H^+$ is independent of the state $(x,\tau)$ and is concave with respect to $p$. 
One may consider the formulation in~\eqref{H+_1} as a perturbation-regularization of the original min-max saddle point problem~\eqref{saddle-one}, where the objective functions $\ell$ are additively perturbed by $E$ and then regularized by $R$. 
\begin{remark} Under~\Cref{assume_JR}, 
by the Fenchel-Young identity~\cite{rockafellar1970convex, bauschke2020correction}, 
\begin{equation}\label{bregman_ind}
    R(u) - \langle p,u \rangle = -R^*(p)  + D_R(u, u_p),
\end{equation}
where $u_p = \nabla R^*(p)$, and $D_R(u,v) = R(u) - R(v) - \nabla R(v)(u-v)$ 
is the \textit{Bregman divergence}. Hence, the Hamiltonian in~\eqref{H+_1} is equivalent to 
\begin{equation}\label{H+_2}
H^+(p,E) = -R^*(p) + \min_{u \in \R^d} \{g\circ(\ell(u) + E)  + D_R(u,u_p)\}.
\end{equation}
 \end{remark}

Recall that, in the original MOO problem~\eqref{saddle-one}, an optimal pair $(u^*,\lambda^*)$ defines a saddle point of the unperturbed Lagrangian 
\begin{equation}
    \Ll(u,\lambda):= \sum_{i=1}^N \lambda_i \ell_i(u) - g^*(\lambda) , \quad u \in \R^d, \ \lambda \in \dom(g^*).
\end{equation}
In other words, each saddle point pair represents one Pareto optimal configuration. 
In the differential game, the argument of $g$ in the Hamiltonian $H^+$~\eqref{H+_1} is shifted by an additive vector $E \in \R^N$, and the extra term $\sum_i \lambda_iE_i$ acts as a linear perturbation of the objectives $\ell_i$.
Hence, varying $E$ leads to different scalarizations and thus, different corresponding Pareto points. 
Formally, we have that $E$ is the dual variable of $\tau$ in the state $(x,\tau)$ of the differential game and HJ system, and infinitesimal changes in $E$ correspond to tangential displacements along the Pareto front. 
Thus, the mapping 
\begin{equation}\label{eq:perturb-lagrangian}
    E \mapsto \inf_u \sup_\lambda \Ll_E(u,\lambda), \quad \Ll_E(u,\lambda):= \Ll(u,\lambda) + \sum_i \lambda_i E_i
\end{equation}
is a parametric perturbation of the base min–max problem~\eqref{saddle-one}. 

Indeed, $R(u)$ introduces a regularization into the inner minimization problem
\begin{equation}\label{eq:inner_min}
    \inf_u \{\langle p, -u\rangle + g \circ (\ell(u)+E) +R(u) \}.
\end{equation}
One natural choice for $R$ is a quadratic, i.e., $R(u) = \frac{\mu}{2}\|u \|^2$.
In this case, the minimizer of~\eqref{eq:inner_min} is unique and continuous in $p$.
Moreover,  as $\mu \to 0$, one recovers the original (possibly non-smooth or nonconvex) saddle system~\eqref{eq:perturb-lagrangian}. Hence, 
$R$ plays the role of a Moreau–Yosida or Bregman regularizer (as in~\eqref{H+_2}) that smooths the Hamiltonian. The term $-R^*(p)$ in~\eqref{H+_2} shifts the Hamiltonian’s baseline concave envelope in $p$. 

\begin{remark}
The lower value function $V^-$ and lower Hamiltonian $H^-$ are related to the convex hull of the Pareto front (i.e., corresponding to switching the order of the $\inf$ and $\sup$ in~\eqref{saddle-one}). Thus, we do not consider $V^-, H^-$ here as they cannot be used to recover nonconvex Pareto fronts. When $\{\ell_i\}_{i\in\{1,2,\dots,N\}}$ are convex, the Pareto front is also convex and we have that $H^+ = H^-$, $V^+ = V^-$.
\end{remark}

\subsection{Hopf-Lax representation of the Pareto front}\label{subsec-hopf-noncons}
 
In this section, we provide a Hopf-Lax formula that yields a representation of the Pareto front, which we will later use to develop a new and efficient MOO algorithm. We adopt classical semigroup evolution notation and consider the 
differential game with upper value function~\eqref{upper_value} and its corresponding HJ equation $F(x,\nabla_{x,\tau} V^+)$.
    Note that under~\Cref{assump_f},~\Cref{assump_g}, and 
    the assumption on $R$ imposed in~\Cref{assume_JR}, 
    the minimum in $H^+$ is attained for all $(p,E) \in \R^d \times \R^N$. 

    For a general proper, lsc, convex terminal cost \(J:\R^d\times\R^N\to\Rb\) with convex conjugate \(J^*\), the viscosity solution to $F(x,\nabla_{x,\tau} V^+)$ is given by the Hopf-Lax  representation formula (e.g., see \cite{hopf1965generalized,lions1986hopf})
    \begin{equation}\label{hopf-upper}
    V^+((x,\tau),\alpha)
    =
    \max_{p\in\R^d,\;E\in\R^N}
    \left\{
    \langle x, p\rangle + \langle \tau, E\rangle  - J^*(p,E) + \alpha\, H^+(p,E)
    \right\}.
    \end{equation}
Fix the initial condition \((x,\tau)\). 
For each \(\alpha\geq0\), let \((\bar{p}_{x,\tau}(\alpha),\bar{E}_{x,\tau}(\alpha))\) be a maximizer of~\eqref{hopf-upper}, and let  \(\bar{u}_{x,\tau}(\alpha)\) be the inner minimizer of \(H^+\) corresponding to \((\bar{p}_{x,\tau}(\alpha),\bar{E}_{x,\tau}(\alpha))\), that is 
\begin{equation}\label{optu}
\bar{u}_{x,\tau}(\alpha)
\in
\argmin_{u\in\R^d}
\left\{
\langle \bar{p}_{x,\tau}(\alpha), -u \rangle
+ g\big(\ell(u) + \bar{E}_{x,\tau}(\alpha)\big)
+ R(u)
\right\}.
\end{equation} 
For each minimizer \(\bar u\), we have
$H^+(\bar p,\bar E)=\langle \bar p,-\bar u\rangle+g(\ell(\bar u)+\bar E)+R(\bar u).$ 
Substituting \(\big(\bar{p}_{x,\tau}(\alpha),\bar{E}_{x,\tau}(\alpha), \bar{u}_{x,\tau}(\alpha)\big)\) into~\eqref{hopf-upper} yields
\begin{equation}\label{hffinal}
\begin{aligned}
V^+&((x,\tau),\alpha)
=
x \cdot \bar{p}_{x,\tau}(\alpha) + \tau \cdot \bar{E}_{x,\tau}(\alpha) - J^*\!\big(\bar p_{x,\tau}(\alpha),\bar E_{x,\tau}(\alpha)\big)
\\
&\qquad 
+ \alpha \Big(
\langle \bar{p}_{x,\tau}(\alpha), -\bar{u}_{x,\tau}(\alpha) \rangle
+ g\big(\ell(\bar{u}_{x,\tau}(\alpha)) + \bar{E}_{x,\tau}(\alpha)\big)
+ R\big(\bar{u}_{x,\tau}(\alpha)\big)
\Big).
\end{aligned}
\end{equation}


Thus, \eqref{hffinal} yields a representation of the Pareto front of the original MOO problem, which we state formally below as the main result of this section.

\begin{theorem}\label{thm:pareto_explore}
Assume \Cref{assump_f}, \Cref{assump_g}, \Cref{assume_JR}. Fix $(x,\tau)$. 
For each $\alpha>0$, define $V^+$ by the Hopf--Lax formula in~\eqref{hffinal} with $(E_\alpha, p_\alpha)$ denoting the maximizer $(E_{x,\alpha}(\alpha), p_{x,\tau}(\alpha))$ in~\eqref{hopf-upper}.
Let $u_\alpha$ denote the minimizer $\bar{u}_{x,\tau}(\alpha)$ in~\eqref{optu}. 
Define $u_{p_\alpha}:=\nabla R^*(p_\alpha)$ and $m(E):=\inf_{u\in\R^d} g(\ell(u)+E)$. 
\begin{enumerate}
\item[(i.)]\label{th1_1}  
The following inequality holds:
\begin{equation}\label{eq:thm-ineq}
    0 \leq g(\ell(u_\alpha) +E_\alpha) - m(E_\alpha) \leq D_R(u_\alpha, u_{p_\alpha}).
\end{equation}
In particular, if $D_R(u_\alpha, u_{p_\alpha}) = 0$, then $u_\alpha \in \argmin_u g(\ell(u) +E_\alpha)$.
\item[(ii.)]\label{th1_2}
Let $\{\alpha_n\}$ be any sequence s.t. $E_{\alpha_n} \to \bar{E}$ and $D_R(u_{\alpha_n}, u_{p_{\alpha_n}}) \to 0$. Any cluster point $\bar{u}_\alpha$ of $\{u_{\alpha_n}\}$ satisfies $\bar{u}_\alpha \in \argmin_u g(\ell(u) +\bar{E}_\alpha)$. Hence, $\ell(\bar{u}_\alpha)\in \F_w(\ell)$. 
\item[(iii.)]\label{th1_3} 
Let $\mathcal{E}:=\{E_\alpha\}_{\alpha>0}$. To each $E\in\mathcal{E}$ associated a  minimizer
$u_E$ of $g(\ell(u)+E)$. Then $\{\ell(u_E)\}_{E\in\mathcal{E}}\subset \F_w(\ell)$. If, in addition, there exists a set $G\subset\R^N$ such that $\mathcal{E}$ is dense in $G$ and the map
$f:G\to\R^N$, $f(E):=\ell(u_E)$ is single-valued and continuous on $G$, then
$\{\ell(u_E)\}_{E\in\mathcal{E}}$ is dense in $\{\ell(u_E)\}_{E\in G}$.
\end{enumerate}
\end{theorem}
\begin{proof}
    Fix $(x,\tau)$ and $\alpha >0$. 
    Recall that $u_{p_\alpha}:=\nabla R^*(p_\alpha)$ and 
\[
D_R(u,v):=R(u)-R(v)-\langle \nabla R(v),u-v\rangle.
\]
Since $R$ is convex and $C^1$, we have that $D_R(u,v)\ge 0$ for all $u,v$.

For (i.), by the Bregman-form identity~\eqref{bregman_ind}, $u_\alpha$ minimizes the map 
$  u\mapsto g\bigl(\ell(u)+E_\alpha\bigr) + D_R\bigl(u,u_{p_\alpha}\bigr). $
Hence,
$ 
g(\ell(u_\alpha)+E_\alpha)+D_R(u_\alpha,u_{p_\alpha})
=\inf_{u}\{g(\ell(u)+E_\alpha)+D_R(u,u_{p_\alpha})\}.
$ 
Using the fact that $D_R\ge 0$, one obtains the following inequality: 
$ 
\inf_{u}\{g(\ell(u)+E_\alpha)+D_R(u,u_{p_\alpha})\}\ge \inf_u g(\ell(u)+E_\alpha)=m(E_\alpha).
$ 
Thus,
\[
g(\ell(u_\alpha)+E_\alpha)-m(E_\alpha)\le D_R(u_\alpha,u_{p_\alpha}).
\]
The lower bound then follows the definition of $m(E_\alpha)$ as the infimum of $g(\ell(u_\alpha)+E_\alpha)$, and hence the two-sided inequality~\eqref{eq:thm-ineq} follows. If $D_R(u_\alpha,u_{p_\alpha})=0$, then
$g(\ell(u_\alpha)+E_\alpha)=m(E_\alpha)$, and therefore
$u_\alpha\in\arg\min_u g(\ell(u)+E_\alpha)$.

For (ii.), let $\{\alpha_n\}$ be a sequence such that $E_{\alpha_n}\to \bar E$ and
$D_R(u_{\alpha_n},u_{p_{\alpha_n}})\to 0$. Denote $u_n:=u_{\alpha_n}$ and $E_n:=E_{\alpha_n}$. 
Let $u^*$ be a minimizer of $g(\ell(u) + \bar{E})$. Since
$ 
m(E_n)=\inf_u g(\ell(u)+E_n)\;\le\; g(\ell(u^*)+E_n),$ 
the result in (i.) implies
\begin{equation}
g\bigl(\ell(u_n)+E_n\bigr)
\;\le\;
g\bigl(\ell(u^*)+E_n\bigr)
+
D_R\bigl(u_n,u_{p_n}\bigr).
\label{eq:key_compare}
\end{equation}
Let $\bar u$ be any cluster point of $\{u_n\}$, that is along any subsequence, then $u_n\to \bar u$. Under the continuity of $(u,E)\mapsto g(\ell(u)+E)$ 
and since $E_n\to\bar E$, we have
\[
g\bigl(\ell(u_n)+E_n\bigr)\to g\bigl(\ell(\bar u)+\bar E\bigr),
\qquad
g\bigl(\ell(u^*)+E_n\bigr)\to g\bigl(\ell(u^*)+\bar E\bigr),
\]
and $D_R(u_n,u_{p_n})\to 0$. Passing to the limit on both sides of~\eqref{eq:key_compare} yields
\[
g(\ell(\bar u)+\bar E)\le g(\ell(u^*)+\bar E)=\inf_u g(\ell(u)+\bar E).
\]
Hence, $\bar u\in \arg\min_u g(\ell(u)+\bar E)$, and $\ell(\bar{u}_\alpha)\in \F_w(\ell)$ follows 
\Cref{prop-onePareto}.

For (iii.), for each $E\in\mathcal E$, any selected minimizer $u_E\in\arg\min_u g(\ell(u)+E)$ is weak Pareto optimal by the same argument  in (ii.). Therefore, $\{\ell(u_E):E\in\mathcal{E}\}\subset \F_w(\ell)$.  
Now suppose $\mathcal{E}$ is dense in some set $G\subset\mathbb{R}^N$ and the map
$ 
f:G\to\mathbb{R}^N, f(E):=\ell(u_E),
$ 
is single-valued and continuous on $G$. 
Since $\mathcal E$ is dense in $G$, we have $\overline{\mathcal E}=G$.  
By continuity of $f$, it holds that $f(\overline{\mathcal E})\subset \overline{f(\mathcal E)}$, and hence
\[
\{\ell(u_E):E\in G\} = f(G)=f(\overline{\mathcal E}) \subset \overline{f(\mathcal E)}
= \overline{\{\ell(u_E):E\in\mathcal E\}},
\]
or, in other words, $\{\ell(u_E):E\in\mathcal E\}$ is dense in $\{\ell(u_E):E\in G\}$.
\end{proof}

\subsection{Numerical exploration by a primal-dual algorithm}\label{sec:unconstrained_algorithm}

We now present an optimization-based primal-dual scheme for numerically exploring the Pareto front via the Hopf-Lax formula. 
Let the terminal cost $J$ and regularizer $R$ be quadratic:
\begin{equation}\label{final_regu}
    J(x,\tau) = \frac{c}{2} (\| x \|^2 + \|\tau\|^2), \quad R(u) = \frac{\mu}{2} \| u\|^2
\end{equation}
for  some $c>0$ and $\mu>0$, where $\|\cdot\|$ denotes the Euclidean norm. In this case, the Hopf-Lax formula~\eqref{hopf-upper} gives 
\begin{equation}\label{hopf-upper-num}
    V^+((x,\tau),\alpha) = \max_{p,E} \left\{\langle x, p \rangle + \langle \tau,E \rangle - \frac{\|p\|^2 + \|E\|^2}{2c} + \alpha H^+(p,E) \right\} .
\end{equation} 
For ease of notation, we omit the dependence on $((x,\tau),\alpha)$, and denote by $(\bar{p},\bar{E})$ the maximizer in~\eqref{hopf-upper-num} and $\bar{u}$ the inner minimizer of $H^+$ that corresponds to $(\bar{p},\bar{E})$. 
Let 
\begin{equation}\label{eq:pi}
     \bar{\pi} \in \partial g(\ell(\bar{u})+ \bar{E}),
\end{equation}
where $\partial$ here denotes the (convex) subdifferential. 
Then, the (Karush–Kuhn–Tucker (KKT), e.g., see~\cite{hiriart1993convex,bauschke2020correction}) optimality conditions for~\eqref{hopf-upper-num} are given by
\begin{equation}\label{optimal_cond}
\left\{
    \begin{aligned}
        & \bar{p} = c(x - \alpha \bar{u}), \\
        & \bar{E} = c(\tau + \alpha \bar{\pi}), \\ 
        & 0 \in - \bar{p} + \partial_u(\bar{\pi} \cdot \ell)(\bar{u}) + \mu \bar{u} \text{ and if } \ell \in \Cc^1, \bar{p} = (\Jac[\ell](\bar{u}))^T \bar{\pi} + \mu \bar{u}, \\
    \end{aligned}
    \right.
\end{equation}
where 
$\Jac[\ell]$ denotes the Jacobian of $\ell$. 
In~\eqref{optimal_cond}, the first two conditions arise from stationarity of the dual objective $(p,E)$, and the third condition is the first order optimality condition of the inner minimization over $u$.

We solve iteratively for $(\bar{u},\bar{\pi})$  and recover $(\bar{p},\bar{E})$ at each iterative step using the first two relations in~\eqref{optimal_cond}. In particular, given an initialization $(u^{(0)},\pi^{(0)})$ and parameters $\rho,\eta >0$, the iteration proceeds as follows: 
\begin{enumerate}
    \item \textit{Dual update for $\pi$.} The multiplier $\pi$ is updated via a proximal step at the current $(u^{(j)},E^{(j)})$ as follows:
    \begin{equation}\label{dual_pi}
        \pi^{(j+1)} = \mathrm{prox}_{\rho g^*}\!\left(
        \pi^{(j)} + \rho\,(\ell(u^{(j)}) + E^{(j)})\right) ,
    \end{equation}
    where $E^{(j)} = c\big(\tau + \alpha \pi^{(j)}\big)$. When only $\mathrm{prox}_g$ is available, one can use the Moreau identity defined by 
    $\mathrm{prox}_{\rho g^*}(v) = v - \rho\,\mathrm{prox}_{g/\rho}(v/\rho)$.
    \item \textit{Primal update for $u$.} The control $u$ is updated using a Levenberg-Marquardt step~\cite{levenberg1944method,marquardt1963algorithm} on the residual of the stationarity condition as follows. In particular, we define the stationarity residual as
    \begin{equation}
        r^{(j)}(u) := \Jac[\ell](u)^\top \pi^{(j+1)} + \mu u - c(x - \alpha u). 
    \end{equation}
    Then, the update reads 
    \begin{equation}
        u^{(j+1)} = u^{(j)} - \eta \,(B^{(j)})^{-1} r^{(j)}(u^{(j)}) ,
    \end{equation}
    where the preconditioner is taken as 
    \begin{equation}
    B^{(j)} = (\mu+\alpha c)I_d + \Jac[\ell](u^{(j)})^\top \Jac[\ell](u^{(j)}). 
    \end{equation}
\end{enumerate}
The above iteration scheme is repeated until both the residual $\|r^{(j)}(u^{(j+1)})\|$ 
and the iterate change $\|\pi^{(j+1)}-\pi^{(j)}\|$ 
are below a prescribed tolerance. At convergence, we set $\bar{u}=u^{(j+1)}$, $\bar{\pi}=\pi^{(j+1)}$ and recover
\begin{equation}
    \bar{p} = c(x - \alpha \bar{u}), 
    \qquad 
    \bar{E} = c(\tau + \alpha \bar{\pi}).
\end{equation}
In~\Cref{alg:PD-KKT}, we give a sketch of the proposed algorithm. 

\begin{algorithm}
\caption{Primal--dual algorithm for the optimality conditions~\eqref{optimal_cond}. }
\label{alg:PD-KKT}
\begin{algorithmic}[1]
\Require $x\in\mathbb{R}^d$, $\tau\in\mathbb{R}^N$, $\alpha>0$, $c>0$, $\mu>0$; 
          $\ell:\mathbb{R}^d\!\to\!\mathbb{R}^N$ with Jacobian $\Jac[\ell]$; convex, increasing $g$ with prox$_g$ (or prox$_{g^*}$); step sizes $(\rho,\eta)$; tolerance $\varepsilon>0$
\State Initialize $u^{(0)} \gets \frac{x}{\max(\alpha, 1)}$, $\pi^{(0)} \gets 0$
\For{$k=0,1,2,\dots$}
  \State \textbf{Dual update on scalarization variable $\pi$:}
  \State $E^{(j)} \gets c\big(\tau + \alpha\,\pi^{(j)}\big)$
  \State $v^{(j)} \gets \pi^{(j)} + \rho\big(\ell(u^{(j)}) + E^{(j)}\big)$
  \State $\pi^{(j+1)} \gets v^{(j)} - \rho\,\mathrm{prox}_{g/\rho}\!\big(v^{(j)}/\rho\big)$ \Comment{Moreau Identity for prox$_{\rho g^*}$}
  \State \textbf{Primal update on control $u$:}
  \State $r^{(j)}(u) \gets \Jac[\ell](u)^\top \pi^{(j+1)} + \mu u - c\,(x - \alpha u)$
  \State $B^{(j)} \gets (\mu+\alpha c)I_d + \Jac[\ell](u^{(j)})^\top \Jac[\ell](u^{(j)})$
  \State Solve $B^{(j)} s^{(j)} = r^{(j)}\!\big(u^{(j)}\big)$ 
  \State $u^{(j+1)} \gets u^{(j)} - \eta\, s^{(j)}$
  \If{$\|r^{(j)}(u^{(j+1)})\|\le \varepsilon$ \textbf{and} $\|\pi^{(j+1)}-\pi^{(j)}\|\le \varepsilon$}
     \State \textbf{break}
  \EndIf
\EndFor
\State \textbf{Set} $u^\star\gets u^{(j+1)}$, $\pi^\star\gets \pi^{(j+1)}$
\end{algorithmic}
\end{algorithm}

\begin{remark}
    The portion of the Pareto optimal set that~\Cref{alg:PD-KKT} can recover is determined by  $\dom(g^*)$. In particular, when $g$ is the weighted-sum function, $\dom(g^*)$ is a singleton, i.e., $|\dom(g^*)|=1$, and the method can recover at most one corresponding Pareto optimal solution. 
    The choice of preference functions $g$ used in the numerical experiments is detailed in~\Cref{sec-numerics}. 
    \end{remark}

\begin{remark}\label{rem:pi-update}
    When $g$ is differentiable, the subdifferential in~\eqref{eq:pi} becomes a singleton, and (if its derivative is easily computable) Lines 5-6 in \Cref{alg:PD-KKT} can be replaced with $\pi^{(j+1)} \gets \nabla g(\ell(u^{(j)})+E^{(j)})$.
\end{remark}

\begin{remark}
    Recall from our discussion of~\eqref{eq:perturb-lagrangian} that different values of $\bar E$ yield different Pareto optimal points. In \Cref{alg:PD-KKT}, we can obtain different values of $\bar E$ by varying $x,\tau,$ and/or $\alpha$.
    Fixing $(x,\tau)$ and varying $\alpha$ is advantageous for computing $V^+$ since $\alpha$ acts a linear scaling on the Hamiltonian in the Hopf-Lax formula~\eqref{hopf-upper-num}. However, in terms of MOO, continuously varying $\alpha$ can lead to discontinuous recovery of the Pareto front due to the nonlinear dependence of $u^\star$ on $\alpha$. Fixing instead $x,\alpha$ and continuously varying $\tau$ generally yields continuous curves along the Pareto front. Indeed, \eqref{optimal_cond} shows that $\tau$ and $\alpha$ affect $\bar E$ only through the combination $\tau + \alpha \bar \pi$, so continuously varying $\tau$ is analagous to continuously varying the scalarization in~\eqref{saddle-one}. Hence, we take this latter approach in our numerical implementations.
    
\end{remark}

\subsubsection{Convergence of the primal-dual algorithm}

To show the convergence of \Cref{alg:PD-KKT}, we first recall the Kurdyka-Łojasiewicz (KL) property, which is a standard tool for proving convergence of nonconvex descent-type algorithms~\cite{attouch2013convergence}. 
Recall that for any symmetric positive-definite matrix $M \in \R^{d\times d}$, the (Mahalanobis) norm $\| \cdot\|_M$ of a vector $r\in \R^d$ induced by $M$ is defined by $\|r \|_M = \sqrt{r^T M r}$.

\begin{definition}[see, e.g., ~\cite{attouch2010proximal, bolte2014proximal}]
    Let $F:\R^m\to\Rb$ be proper and lsc, and let $\bar z\in\dom(\partial F)$.
We say that $F$ has the \textit{KL property at $\bar z$} if there exist $\eta\in(0,+\infty]$, a neighborhood $U$ of $\bar z$, and a concave function $\varphi:[0,\eta)\to\R_+$ that is
continuous on $[0,\eta)$ and $C^1$ on $(0,\eta)$ with $\varphi(0)=0$ and $\varphi'(s)>0$ for $s>0$, such that
for all $z\in U$ with $F(\bar z)<F(z)<F(\bar z)+\eta$,
\[
\varphi'\!\big(F(z)-F(\bar z)\big)\,\mathrm{dist}\!\big(0,\partial F(z)\big)\ \ge\ 1.
\]
If $\varphi(s)=c\,s^{1-\theta}$ for some $\theta\in[0,1)$, we say that $F$ has\textit{ KL exponent $\theta$ at $\bar z$}.
\end{definition}

We now introduce a merit function that measures violation of the primal-dual optimality system~\eqref{optimal_cond}. 
Throughout this subsection, we fix $(x,\tau)\in\R^d\times\R^N$ and $\alpha,c,\mu>0$ and denote
\[
E(u,\pi):=c\big(\tau+\alpha\,\pi\big),
\qquad
B(u):=(\mu+\alpha c)I_d+\Jac[\ell](u)^\top\Jac[\ell](u).
\]
Note that $B(u)\succ 0$ for all $u$ since $\mu+\alpha c>0$.

\begin{lemma}\label{lem:definability_closure}
Assume $g:\R^N\to\overline{\R}$ is proper, lsc, convex, and definable in an o-minimal structure \cite{bolte2007clarke,bolte2014proximal}.
Then,
\begin{enumerate}
\item the Legendre-Fenchel conjugate $g^*$ is definable;
\item for any $\rho>0$, the proximal mapping $\prox_{\rho g^*}$ is definable.
\end{enumerate}
Moreover, if $u\mapsto B(u)$ is definable and $B(u)\succ 0 \ \forall u$, then $u\mapsto B(u)^{-1}$ is definable.
\end{lemma}
\begin{proof}
    Statements (1) and (2) are standard consequences of o-minimal closure properties. In particular, conjugation and argmin mappings of strongly convex, definable functions preserve definability. Moreover, $\prox_{\rho g^*}(v)$ is the unique minimizer of $\pi\mapsto g^*(\pi)+\frac{1}{2\rho}\|\pi-v\|^2$.    
    For the last statement, note that matrix inversion is a rational operation in the matrix entries and hence is definable on the set where $\det(B(u))\neq 0$, which here is all of $\R^d$ since $B(u)\succ 0$.
\end{proof}

\begin{lemma}\label{lem:psi_KL}
Assume~\Cref{assump_f} and \Cref{assump_g}. Assume further that $g$ is semi-algebraic (or, more generally, definable in an o-minimal structure~\cite{bolte2007clarke,attouch2013convergence}) (e.g., $g$ is a weighted max, weighted sum, or log-sum-exp), and let $\ell:\R^d\to\R^N$ be $\Cc^1$ with semi-algebraic (or more generally, definable) graph (e.g., polynomials, rational functions, $\log\!\sum\exp$). 
Fix $x\in\R^d$, $\tau\in\R^N$, $\alpha,c,\mu>0$, and define for $(u,\pi)\in\R^d\times\R^N$
\begin{equation}\label{merit}
\begin{adjustbox}{width=\textwidth}
$\Psi(u,\pi)
:= \frac{1}{2}\|\Jac[\ell](u)^\top \pi + \mu u - c(x-\alpha u)\|_{B(u)^{-1}}^2+ \frac{1}{2\rho^2}\,\|\operatorname{prox}_{\rho g^*}\big(\pi+\rho(\ell(u)+E(u,\pi))\big)-\pi\|^2.
$
\end{adjustbox}
\end{equation}
Then, $\Psi$ is proper, lsc, and definable. In particular, $\Psi$ satisfies the KL property at every point of its domain. 
\end{lemma}
\begin{proof}
    Since $\ell$ is $\Cc^1$ with definable graph, both $\ell$ and $\Jac[\ell]$ are also definable. Thus, the map $u\mapsto B(u)$ is definable and $B(u)\succ 0$, so $u\mapsto B(u)^{-1}$ is definable by~\Cref{lem:definability_closure}. Hence, the first term in \eqref{merit} is definable and continuous. 

    By Lemma~\ref{lem:definability_closure}, $g^*$ and $\prox_{\rho g^*}$ are definable, and since
$E(u,\pi)=c(\tau+\alpha\pi)$ is affine, the composition
$(u,\pi)\mapsto\prox_{\rho g^*}\big(\pi+\rho(\ell(u)+E(u,\pi))\big)$ is also definable. 
Therefore, $\Psi$ is a sum of definable continuous functions and is in turn definable and lsc (indeed, it is continuous). 
$\Psi$ is proper  because $\Psi\ge 0$ and is finite everywhere. Finally, any proper, lsc, definable function satisfies the KL property at every point of its domain (e.g., see 
\cite{bolte2007clarke,attouch2013convergence}). 
\end{proof}

We place the following additional assumption on $g$.
\begin{assumption}\label{ass:bounded_dom_gstar}
The domain $\dom(g^*)$ is bounded.
\end{assumption}

\begin{lemma}\label{lem:pi_bounded}
Under Assumption~\ref{ass:bounded_dom_gstar}, we have that the proximal update
$\pi^{(j+1)}=\prox_{\rho g^*}(v^{(j)})$ satisfies $\pi^{(j)}\in\dom(g^*)$ for all $j$. Hence, $\{\pi^{(j)}\}$ is bounded.
\end{lemma}

\begin{lemma}\label{lem:fixedpoint_optcond}
Let $(u,\pi)\in\R^d\times\R^N$. Then, the following are equivalent:
\begin{enumerate}
\item[(i)] $(u,\pi)$ satisfies the optimality system~\eqref{optimal_cond};
\item[(ii)] $\pi=\prox_{\rho g^*}\big(\pi+\rho(\ell(u)+E(u,\pi))\big)$ and the corresponding stationarity residual in
\eqref{optimal_cond} vanishes.
\end{enumerate}
In particular, $\Psi(u,\pi)=0$ implies \eqref{optimal_cond}.
\end{lemma} 
\begin{proof}
    The proximal fixed-point identity is equivalent to the inclusion
$\ell(u)+E(u,\pi)\in\partial g^*(\pi)$, which is equivalent to $\pi\in\partial g(\ell(u)+E(u,\pi))$
by conjugacy of subdifferentials for proper, lsc, convex functions. 
Together with the vanishing of the stationarity residual, this relation is exactly equivalent to \eqref{optimal_cond}. 
The last statement follows since $\Psi(u,\pi)=0$ forces both squared residual terms in \eqref{merit} to be zero. 
\end{proof}

\begin{lemma}\label{lem:abstract_KL}
Let $\Psi:\R^m\to\overline{\R}$ be proper and lsc and satisfy the KL property.
Let $\{z^{(j)}\}\subset\R^m$ be bounded, and assume that there exist constants $a,b>0$ such that:
\begin{enumerate}
\item[(H1)] $\Psi(z^{(j+1)})\le \Psi(z^{(j)})-a\|z^{(j+1)}-z^{(j)}\|^2$ for all $j$;
\item[(H2)] $\dist(0,\partial\Psi(z^{(j+1)}))\le b\|z^{(j+1)}-z^{(j)}\|$ for all $j$.
\end{enumerate}
Then $\sum_k\|z^{(j+1)}-z^{(j)}\|<\infty$ and $z^{(j)}\to \bar z$ for some critical point $\bar z$.
If $\Psi$ has KL exponent $\theta\in[0,1)$ at $\bar z$, then the standard KL rates hold.
\end{lemma}

The convergence of \Cref{alg:PD-KKT} then follows from standard convergence analysis techniques for primal-dual algorithms (e.g., see \cite{combettes2012primal,attouch2013convergence,wright2022optimization}), which we state below as the main result of this section.

\begin{theorem}
    Assume~\Cref{assump_g}, \Cref{assump_f}, \Cref{assume_JR}, \Cref{ass:bounded_dom_gstar}, and that
    \begin{equation}\label{eq:jac_bounds}
\|\Jac[\ell](u)\|_2 \le L_J \quad\text{ and }\quad \|\Jac[\ell](u)-\Jac[\ell](v)\|_2 \le L_H\|u-v\|, \quad \forall u,v\in\R^d.
\end{equation}
Consider the primal--dual iteration of \Cref{alg:PD-KKT} with parameters $(\rho,\eta)$ satisfying
\begin{equation}\label{eq:stepsizes}
0<\rho<\frac{1}{L_J^2}
\quad\text{and}\quad
0<\eta\le 1.
\end{equation} 
Let $\Psi$ be defined by \eqref{merit}, and set $z^{(j)}:=(u^{(j)},\pi^{(j)})$. 

Further assume that the iteration satisfies the standard descent conditions (H1)--(H2) of
Lemma~\ref{lem:abstract_KL} for the merit function $\Psi$
(this assumption can be enforced, for instance, by a simple backtracking safeguard on the damping parameter $\eta$). Then,
\begin{enumerate}
    \item[(i.)] The sequence $\{z^{(j)}\}_{j\ge0}$ is well-defined and bounded.
Moreover, $\Psi(z^{(j)})$ 
decreases with each iteration. 
\item[(ii.)] Every cluster point $(\bar u,\bar\pi)$ of $\{z^{(j)}\}$ satisfies the optimality system~\eqref{optimal_cond}; i.e., all limit points are first-order stationary. 
\item[(iii.)]  The full sequence $z^{(j)}$ converges to a single limit $(\bar u,\bar\pi)$. 
\item[(iv.)] If, in addition, $\Psi$ has KL exponent $\theta\in[0,1)$ at $(\bar u,\bar\pi)$, then the standard rates hold:
finite length for $\theta=0$, linear for $\theta=\tfrac12$, and sublinear $O(j^{-\frac{1-\theta}{2\theta-1}})$ for $\theta\in(1/2,1)$.
\end{enumerate}
\end{theorem}
\begin{proof}
For (i.), 
    the proximal update for $\pi^{(j+1)}$ is well-defined and single-valued since $g^*$ is proper, lsc, and convex. 
    By Lemma~\ref{lem:pi_bounded}, $\{\pi^{(j)}\}$ is bounded. 
    Moreover, $B(u)\succeq (\mu+\alpha c)I_d$, so the Gauss-Newton direction in \Cref{alg:PD-KKT} is well-defined. 
    Since $\Psi(z^{(j)})$ decreases and $\Psi\ge 0$, the values $\Psi(z^{(j)})$ stay bounded for all $j$. 
    The first term in \eqref{merit} is coercive in $u$ on bounded $\pi$-sets because the residual contains
$(\mu+\alpha c)u$ and $B(u)^{-1}$ is uniformly bounded from above. Hence, $\{u^{(j)}\}$ is bounded, which proves that $\{z^{(j)}\}$ is bounded. $\Psi(z^k)$ decreasing in $k$ is exactly (H1) in \Cref{lem:abstract_KL}.  

For (ii.), let $(\bar u,\bar\pi)$ be any cluster point of $\{z^{(j)}\}$. 
Since $\|z^{(j+1)}-z^{(j)}\|\to 0$ and the update map in \Cref{alg:PD-KKT} is continuous under \eqref{eq:jac_bounds},
passing to the limit in the fixed-point form of the iteration yields that $(\bar u,\bar\pi)$ is a fixed point.
By Lemma~\ref{lem:fixedpoint_optcond}, any such fixed point satisfies \eqref{optimal_cond}, which proves (ii.).

For (iii.) and (iv.), 
by Lemma~\ref{lem:psi_KL}, $\Psi$ satisfies the KL property.
Together with boundedness and (H1)-(H2), Lemma~\ref{lem:abstract_KL} yields that
$\sum_j\|z^{(j+1)}-z^{(j)}\|<\infty$ and $z^{(j)}\to(\bar u,\bar\pi)$ for some critical point, which proves (iii.). (iv) follows from the KL exponent rates in Lemma~\ref{lem:abstract_KL}. 
\end{proof}

\section{Extension to constrained MOO}\label{sec-con}
State constraints are required in most real-world applications but can complicate MOO.  In this section, we show that the approach and algorithm from the previous section have straightforward extensions to MOO problems~\eqref{MOO_obj} with inequality constraints, i.e., problems of the form 
\begin{equation}\label{eq:constrained_moo}
\begin{aligned}
    & \min_{u} \ell(u) : = (\ell_1(u),\ell_2(u),\dots, \ell_N(u)) \\ 
    & \ \text{s.t. } \ k_i(u) \geq0 , \ i = 1,2,\dots, m,
\end{aligned}
\end{equation}
where $k_i:\R^d\to\R$ is continuously differentiable $\forall i$. We denote the constraint set by
\begin{equation}
    K : = \{ u \in \R^d \mid k_i(u) \geq0 , \ i = 1,2,\dots, m  \}. 
\end{equation} 
Throughout this section, we assume that $K$ is nonempty, closed, compact, and convex and that standard constraint qualifications (e.g., Slater’s condition~\cite{rockafellar1970convex,BoydVand2004}) hold, so that normal cones and KKT conditions behave in the usual way.

Similarly to Section~\ref{sec:fixedpt}, one Pareto optimal solution to~\eqref{eq:constrained_moo} can be obtained by applying a preference function $g$ but now subject to $u \in K$. 
A standard convex analysis technique for enforcing the constraints $K$ is to use the convex indicator function $I_K$ defined by $I_K(u) = 0$ if $u\in K$ and $I_K(u) = +\infty$ otherwise.
Adding $I_K(u)$ to the objective function guarantees that infeasible $u$ never contribute. Hence, one Pareto optimal solution to the constrained MOO problem~\eqref{eq:constrained_moo} can be found by solving the following classical optimization problem:
\begin{equation}
    \min_{u \in \R^d} g \circ \ell(u) + I_K(u).
\end{equation}
Similarly, the corresponding fixed point formulation is identical to that in~\eqref{saddle-one} but with an additional $+I_K(u)$ term.

\subsection{Differential game formulation}

We now extend the associated differential game formulation and Hopf–Lax representation to the constrained MOO case. 
Let $\Uu_K$ denote the set of admissible controls with values in $K$, that is,
       $ \Uu_K:= \{u:[\alpha,\beta] \to K \mid u(\cdot) \text{ is measurable} \}, $
and keep the same set~$\Lambda$ of dual trajectories as in the unconstrained setting. Then, we define the upper value function as
\begin{equation}
\begin{adjustbox}{width=\textwidth}
$    V^+_K((x,\tau),\alpha) :=  \sup_{\lambda \in \Lambda} \inf_{u \in \Uu_K}\Big\{\int_{0}^{\alpha} \Big( \sum_{i=1}^N \lambda_i(s) \ell_i(u(s)) -g^*(\lambda(s)) + R(u(s)) \Big) ds + J(x(\alpha),\tau(\alpha))  \Big\},
$
\end{adjustbox}
\end{equation} 
where the dynamics for $(x(\cdot),\tau(\cdot))$ are the same as in~\eqref{dynamics}. Note that the only change from the unconstrained case is that the control $u$ is now constrained to lie in $\Uu_K$. 
    
The associated upper Hamiltonian with constraints is given by
\begin{equation}\label{hamiltonian_k}
\begin{aligned}
    H^+_K(p,E):
&=
\min_{u\in K} \Big\{ \langle p,-u\rangle + g\circ\big(\ell(u)+E\big) + R(u)\Big\} \\ &= 
-\big(g\circ(\ell(\cdot)+E)+ R+I_K\big)^*(p) \ 
\end{aligned}
\end{equation}
for $(p,E)\in\R^d\times\R^N$, where $I_K$ is the indicator of $K$. 
Under \Cref{assump_f}, \Cref{assump_g}, \Cref{assume_JR}, and the above assumptions on~$K$, the minimum in~\eqref{hamiltonian_k} is attained for all $(p,E)$, the map $H^+_K$ is continuous and has at most linear growth,  and thus the HJI equation $F(x, \nabla V^+_K)$
admits a unique viscosity solution, which coincides with $V^+_K$ by the dynamic programming principle, exactly as in the unconstrained case.
As in~\Cref{subsec-hopf-noncons}, for a general proper, lsc, convex terminal cost $J:\R^d\times\R^N\to\Rb$ with convex conjugate $J^*$, the value function can be represented by the Hopf-Lax formula
\begin{equation}\label{value_con}
   V^+_K((x,\tau),\alpha)
    =
    \max_{p\in\R^d,\;E\in\R^N}
    \left\{
    \langle x, p\rangle + \langle \tau, E\rangle  - J^*(p,E) + \alpha\, H^+_K(p,E)
    \right\},
\end{equation}
which is the constrained analogue of~\eqref{hopf-upper}, obtained by replacing $H^+$ with $H^+_K$. 

\subsection{Optimality conditions and a primal-dual algorithm}
In this section, we present the optimality conditions for the finite-dimensional saddle point problem in~\eqref{value_con}. We then provide a corresponding primal-dual algorithm for numerically computing the solution that mirrors the unconstrained scheme in~\Cref{alg:PD-KKT}.

As in~\Cref{sec:unconstrained_algorithm}, we choose a quadratic terminal cost and quadratic regularizer, as in~\eqref{final_regu}. 
Then, we have $J^*(p,E) = \frac{1}{2c}(\|p\|^2+\|E\|^2)$, and~\eqref{value_con} reads
\begin{equation}\label{eq:hopf-upper-constrained-num}
    V^+_K((x,\tau),\alpha)
    = \max_{p\in\R^d,\;E\in\R^N}
    \Big\{
      \langle x,p\rangle + \langle \tau,E\rangle
      - \tfrac{1}{2c}(\|p\|^2+\|E\|^2)
      + \alpha\,H^+_K(p,E)
    \Big\}.
\end{equation}
For each fixed $(x,\tau,\alpha)$, let $(\bar p_K,\bar E_K)$ be a maximizer of~\eqref{eq:hopf-upper-constrained-num}, and let
\begin{equation}\label{eq:constrained_innermin}
    \bar u_K \in \argmin_{u\in K}
    \Big\{
      \langle \bar p_K,-u\rangle
      + g\big(\ell(u)+\bar E_K\big)
      + R(u)
    \Big\},
\end{equation}
so that $\bar u_K$ is a minimizer of the constrained Hamiltonian $H^+_K(\bar p_K,\bar E_K)$. Let
\begin{equation}\label{eq:constrain_pi}
    \bar\pi_K \in \partial g(\ell(\bar u_K)+\bar E_K),
\end{equation}
where $\partial$ denotes the convex subdifferential. As in the unconstrained case, the stationarity of~\eqref{eq:hopf-upper-constrained-num} in $(p,E)$ yields
\begin{equation}\label{eq:pE-constrained}
    \bar{p}_K = c\,(x-\alpha \bar u_K),
    \qquad
    \bar{E}_K = c\,(\tau + \alpha \bar\pi_K),
\end{equation}
which coincides with the first two relations in~\eqref{optimal_cond}. 

The stationarity of the inner minimization over $u$ in~\eqref{hamiltonian_k} gives
\begin{equation}\label{station_cons-general}
    0 \in -\bar p_K + \partial_u\big(\bar\pi_K\cdot\ell\big)(\bar u_K)
          + \partial R(\bar u_K) + N_K(\bar u_K),
\end{equation}
where $N_K(\bar u_K)$ is the Euclidean normal cone to $K$ at $\bar u_K$. Under the smoothness assumptions on $\ell$ (\Cref{assump_f}) and our choice of quadratic $R$, \eqref{station_cons-general} reduces to
\begin{equation}\label{station_cons}
    0 = -\bar p_K
        + \Jac[\ell](\bar u_K)^\top \bar\pi_K
        + \mu \bar u_K
        + v_K,
    \qquad
    v_K \in N_K(\bar u_K).
\end{equation} 
Substituting~\eqref{eq:pE-constrained} into~\eqref{station_cons}, it is convenient to define the residual map
   $ F(u) := \Jac[\ell](u)^\top \bar\pi_K + \mu u - c\,(x-\alpha u)$, 
so that~\eqref{station_cons} becomes
    $0 \in F(\bar u_K) + N_K(\bar u_K)$. 
By standard convex analysis, this inclusion is equivalent to the variational inequality
\begin{equation}\label{eq:VI}
    \text{Find }\bar u_K\in K,\ \text{s.t.}\ \
    \langle F(\bar u_K),v-\bar u_K\rangle \ge 0,
    \quad \forall\, v\in K.
\end{equation} 
Define $k:\R^d\to \R^m$ by $k := (k_1, k_2,\dots, k_m)$.  Then, introducing inequality multipliers $\nu\in\R^m_+$ for the constraints $k_i(u)\ge0$, the normal cone admits the representation
\[
N_K(\bar u_K)
=
\big\{
  \Jac[k](\bar u_K)^\top \nu
  : \nu\ge0,\ \nu^\top k(\bar u_K)=0
\big\},
\]
and the first-order (KKT) conditions become
\begin{equation}\label{eq:KKT-constr}
\left\{
\begin{aligned}
  &0 =
    \Jac[\ell](\bar u_K)^\top \bar\pi_K
    + \mu \bar u_K
    - \Jac[k](\bar u_K)^\top \bar\nu
    - c\,(x-\alpha \bar u_K),\\[2pt]
  &\bar\nu \ge 0,\quad k(\bar u_K)\ge0,\quad
   \bar\nu^\top k(\bar u_K)=0,
\end{aligned}
\right.
\end{equation}
together with~\eqref{eq:constrain_pi} and~\eqref{eq:pE-constrained}, 
which is the natural constrained counterpart of~\eqref{optimal_cond}.

We now use a simple primal-dual iteration that mirrors the unconstrained scheme but adds a projected dual ascent for the multipliers~$\nu$. Given $(u^{(0)},\pi^{(0)},\nu^{(0)})$ and stepsizes $\rho,\sigma,\eta>0$, one iteration reads:
\begin{enumerate}
\item \textit{Dual update on $\pi$.} As in the unconstrained case, we update $\pi$ using~\eqref{dual_pi}. 
\item \textit{Dual update on $\nu$ (inequality multipliers).} We update $\nu$ via projected gradient ascent in the constraint channel:
      \[
      \nu^{(j+\frac12)} = \nu^{(j)} + \sigma(-k(u^{(j)})),\qquad
      \nu^{(j+1)} = [\nu^{(j+\frac12)}]_+,
      \]
      where $[\cdot]_+$ denotes componentwise projection onto $\R^m_+$.

\item \textit{Primal (Levenberg--Marquardt) update on $u$.} Define
      \[
      r^{(j)}(u) := \Jac[\ell](u)^\top \pi^{(j+1)}
                 - \Jac[k](u)^\top \nu^{(j+1)}
                 + \mu u
                 - c\,(x-\alpha u),
      \]
      \begin{equation}\label{eq:constrained_precond}
      \begin{adjustbox}{width=.8\textwidth}
      $      B^{(j)} := (\mu+\alpha c)I_d
                 + \Jac[\ell](u^{(j)})^\top \Jac[\ell](u^{(j)})
                 + \Jac[k](u^{(j)})^\top W^{(j)} \Jac[k](u^{(j)}),
      $
      \end{adjustbox}
      \end{equation}
      where $W^{(j)}\succeq 0$ emphasizes nearly active constraints (e.g., $W^{(j)}$ is a diagonal matrix built from indicators of $\{i: k_i(u^{(j)})$ is close to $0\}$). We then compute $s^{(j)}$ using
      $B^{(j)} s^{(j)} = r^{(j)}$ and set
      \[
      \tilde u^{(j+1)} = u^{(j)} - \eta\, s^{(j)},\qquad
      u^{(j+1)} =
      \begin{cases}
        \Pi_K(\tilde u^{(j+1)}), & \text{if }\Pi_K\text{ is available},\\
        \tilde u^{(j+1)}, & \text{otherwise}
      \end{cases}
      \]
      where $\Pi_K$ denotes projection onto $K$.
\end{enumerate}

In \Cref{alg:PD-KKT-constr}, we summarize our primal-dual scheme that extends~\Cref{alg:PD-KKT} to the constrained case by adding a projected ascent step on the inequality multipliers and (optionally) a projection onto $K$.

\begin{algorithm}
\caption{Primal--dual scheme for constrained KKT system~\eqref{eq:KKT-constr}.}
\label{alg:PD-KKT-constr}
\begin{algorithmic}[1]
\Require $x\in\R^d$, $\tau\in\R^N$, $\alpha>0$, $c>0$, $\mu>0$; $\ell,k$ with Jacobians $\Jac[\ell],\Jac[k]$;
          $g$ with prox$_g$ (or prox$_{g^*}$);
         step sizes $(\rho,\sigma,\eta)>0$; tolerance $\varepsilon>0$
\State Choose $u^{(0)}\in K$ (e.g., $\Pi_K \frac{x}{\max(\alpha,1)}$), set $\pi^{(0)}\gets 0$, $\nu^{(0)}\gets 0$
\For{$j=0,1,2,\dots$}
  \State \textbf{Dual update on $\pi$:} same as Lines 4-6 in~\Cref{alg:PD-KKT}
  \State \textbf{Dual update on inequality multipliers $\nu$:}
  \State $\nu^{(j+\frac12)} \gets \nu^{(j)} + \sigma\,(-k(u^{(j)}))$
  \State $\nu^{(j+1)} \gets [\nu^{(j+\frac12)}]_+$ \Comment{projection onto $\R^m_+$ (componentwise)}
  \State \textbf{Primal update on control $u$:}
  \State $r^{(j)}(u) \gets
        \Jac[\ell](u)^\top \pi^{(j+1)}
        - \Jac[k](u)^\top \nu^{(j+1)}
        + \mu u
        - c\,(x-\alpha u)$
  \State $B^{(j)} \gets
        (\mu+\alpha c)I_d
        + \Jac[\ell](u^{(j)})^\top \Jac[\ell](u^{(j)})
        + \Jac[k](u^{(j)})^\top W^{(j)} \Jac[k](u^{(j)})$
  \State \Comment{$W^{(j)}\succeq0$ emphasizes nearly active constraints}
  \State Solve $B^{(j)} s^{(j)} = r^{(j)}(u^{(j)})$ 
  \State $\tilde u^{(j+1)} \gets u^{(j)} - \eta\, s^{(j)}$
  \State $u^{(j+1)} \gets
    \begin{cases}
      \Pi_K(\tilde u^{(j+1)}), &\text{if a cheap projector }\Pi_K\text{ is available},\\
      \tilde u^{(j+1)}, &\text{otherwise}
    \end{cases}$
  \If{$\|r^{(j)}(u^{(j+1)})\|\le\varepsilon$ \textbf{and}
      $\|\pi^{(j+1)}-\pi^{(j)}\|\le\varepsilon$ \textbf{and}
      $\|\nu^{(j+1)}-\nu^{(j)}\|\le\varepsilon$}
    \State \textbf{break}
  \EndIf
\EndFor
\State \textbf{Set} $u^\star\gets u^{(j+1)}$, $\pi^\star\gets \pi^{(j+1)}$, $\nu^\star\gets\nu^{(j+1)}$
\end{algorithmic}
\end{algorithm}

\begin{remark}\label{rem:u-update}
    When $K$ admits an efficient projector $\Pi_K$ and $g, \ell, R$ are sufficiently smooth, other numerical methods may be used to compute~\eqref{eq:constrained_innermin} directly. In such cases, \Cref{alg:PD-KKT} can be applied to these constrained problems but with Lines 9-11 replaced by $u^{(j+1)} \gets \argmin_{u\in K}
    \Big\{
      \langle  p^{(j)},-u^{(j)}\rangle
      + g\big(\ell(u^{(j)})+E^{(j)}\big)
      + R(u^{(j)})
    \Big\}$.
\end{remark}

We now briefly state a convergence result for the constrained
primal-dual~\Cref{alg:PD-KKT-constr}. The analysis follows the same structure as in the
unconstrained case. As such, we only sketch the main steps and omit the proof. 

\begin{lemma}
Assume~\Cref{assump_f}, \Cref{assump_g}. 
Assume $g$ is semi-algebraic (more generally, definable in an o-minimal structure)
and that each constraint function
$k_j$ is $\Cc^1$ and semi-algebraic (more generally, definable in an
o-minimal structure). Assume further that $\ell$ is $\Cc^1$ with definable graph. 
Fix $x\in\R^d$, $\tau\in\R^N$,
$\alpha,c,\mu>0$, and define the constrained merit function
$\Psi_K:\R^d\times\R^N\times\R^m\to\Rb$ by
\[
\begin{adjustbox}{width=\textwidth}
$\Psi_K(u,\pi,\nu)
:= \frac12\big\|r(u,\pi,\nu)\big\|_{B(u)^{-1}}^2 + \frac{1}{2\rho^2}\big\|
   \prox_{\rho g^*}\big(\pi+\rho(\ell(u)+E(u,\pi))\big)-\pi
   \big\|^2   + \frac{1}{2\sigma^2}\big\|
   [\nu+\sigma(-k(u))]_+ - \nu
   \big\|^2,$
\end{adjustbox}\]
where
\[
r(u,\pi,\nu)
:= \Jac[\ell](u)^\top \pi
   - \Jac[k](u)^\top \nu
   + \mu u - c(x-\alpha u) ,
\]
\[
B(u):=(\mu+\alpha c)I_d+\Jac[\ell](u)^\top\Jac[\ell](u),\qquad
E(u,\pi)=c(\tau+\alpha\pi),
\]
and $[\cdot]_+$ denotes the projection onto $\R_+^m$.
Then, $\Psi_K$ is proper, lsc, and definable. In particular, $\Psi_K$
satisfies the KL property at every point of its
domain.
\end{lemma}

\begin{theorem}
Assume~\Cref{assump_f}, \Cref{assump_g}, and that $\ell$ and $k$ are $\Cc^1$ and have bounded, Lipschitz Jacobians on $\R^d$. Fix $x,\tau,\alpha,c,\mu$ as above.
Let $\{(u^{(j)},\pi^{(j)},\nu^{(j)})\}_{j\ge0}$ be the sequence
generated by the constrained primal--dual iteration
of~\Cref{alg:PD-KKT-constr} with sufficiently small step  sizes $(\rho,\sigma,\eta)$ (e.g., $0<\rho<\bar\rho$, $0<\sigma<\bar\sigma$, $0<\eta\le1$
for suitable constants $\bar\rho,\bar\sigma>0$ depending only on the
Lipschitz moduli). 
Further assume that the iteration satisfies standard descent conditions (e.g., (H1)--(H2) of
Lemma~\ref{lem:abstract_KL}). 
Then,
\begin{enumerate}
\item[i.] The merit function $\Psi_K(u^{(j)},\pi^{(j)},\nu^{(j)})$
  decreases with each iteration. Moreover,  
  $\|z^{(j+1)}-z^{(j)}\|\to 0$, where
$z^{(j)}:=(u^{(j)},\pi^{(j)},\nu^{(j)})$. 
\item[ii.] Every cluster point
  $(\bar u,\bar\pi,\bar\nu)$ is a first-order
  stationary point of the constrained problem, i.e., it satisfies the
  KKT system~\eqref{eq:KKT-constr} together with~\eqref{eq:pE-constrained}.
\item[ii.] If $\Psi_K$ has KL exponent $\theta\in[0,1)$ at $(\bar u,\bar\pi,\bar\nu)$, then the standard KL rates hold
(finite length for $\theta=0$, linear for $\theta=\tfrac12$, and sublinear otherwise). 
\end{enumerate}
\end{theorem}

\section{Numerical experiments}\label{sec-numerics}

In this section, we demonstrate the capabilities of our algorithms in exploring the Pareto front in several complex MOO settings, including those involving semi-algebraic constrained optimization, nonconvex objective functions, and/or high dimensions in both the dimension $d$ of the domain and the dimension $N$ of the objective function. 
In all of the experiments, we take the soft-max (i.e., entropy-type) preference function $g$ of the form
\begin{equation}
     g(y) \;=\; \varepsilon_\ell \log\!\Big(\sum_{i=1}^N e^{\frac{y_i}{\varepsilon_\ell}}\Big)
\end{equation}
with $\epsilon_\ell  = 0.1$. 
Note that this choice of $g$ satisfies~\Cref{assump_g} and is differentiable with explicitly computable derivative. As such, in all experiments, we update $\pi$ in~\Cref{alg:PD-KKT} and~\Cref{alg:PD-KKT-constr} according to Remark \ref{rem:pi-update}.
The terminal cost $J$ and regularizer $R$ are chosen  as in~\eqref{final_regu}, that is
\begin{equation}
    J(x,\tau) = \frac{c}{2} \| (x,\tau) \|^2, \quad R(u) = \frac{\mu}{2} \| u\|^2. 
\end{equation} 
All numerics were implemented in \textsc{Matlab} and run on an Apple M3 Pro CPU. 

\subsection{Example 1. Constrained semi-algebraic optimization}\label{example1}

We first solve a 2D constrained MOO problem with semi-algebraic feasible set, inspired by the examples in~\cite{magron2014approximating}. 
Consider the objective function $\ell = \{\ell_1, \ell_2 \} \in \R^2$, where 
\begin{equation}
    \ell_1(u) =  - u_1, \quad \ell_2(u) = u_1 + u_2^2
\end{equation}
with polynomial constraints $k_1, k_2$ given by
\begin{equation}
    k_1(u) = - u_1^2 + u_2, \quad k_2(u) = -u_1-2 u_2 + 3.
\end{equation}
Then, the feasible set is defined by the intersection of two polynomial inequalities 
\begin{equation}
    K = \{ u \in \R^2 \mid k_1(u) \geq0, k_2(u) \geq 0 \}.
\end{equation}
\begin{figure}
    \centering
    \subfigure[Discovered Pareto optimal solutions within the feasible set.]{
    \includegraphics[width=0.4\linewidth]{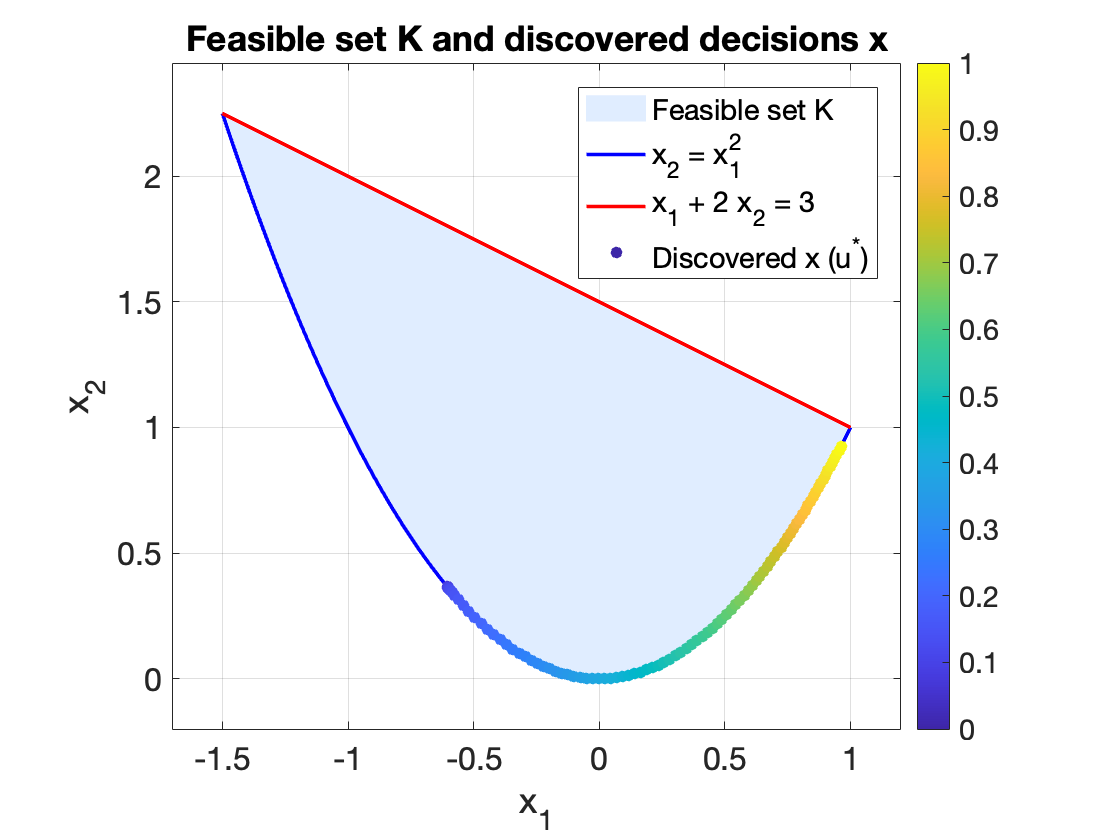}}
    \hfill
    \subfigure[Pareto front recovered by our HJ/Hopf-lax solver.]{
    \includegraphics[width=0.4\linewidth]{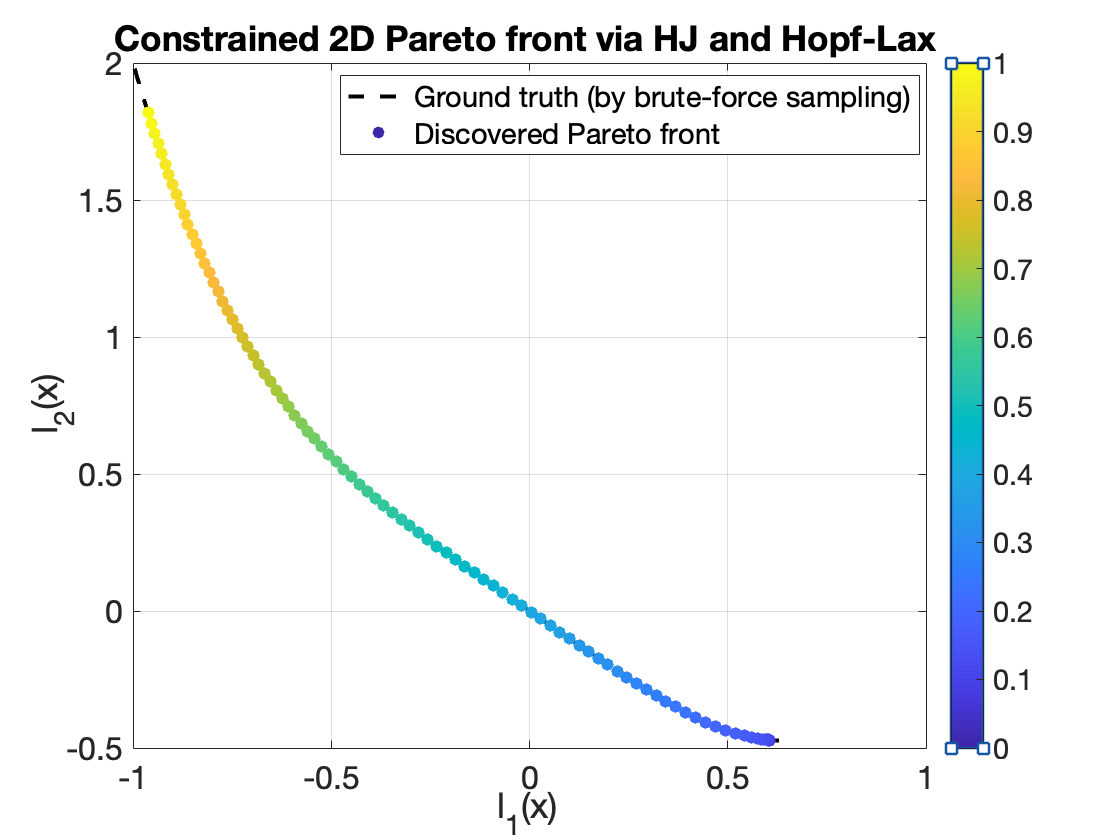}}
    \caption{\textit{Constrained semi-algebraic MOO problem.} Our primal-dual HJ/Hopf-Lax algorithm ($\cdot$) exactly recovers the ground truth Pareto front ($--$) of a 2D constrained convex MOO problem with semi-algebraic feasible set in less than 0.1s. In particular, all discovered Pareto optimal solutions lie on the boundary of the feasible set. }
    \label{fig:case1}
\end{figure}

In \Cref{fig:case1}b, we see that our algorithm perfectly recovers the ground truth Pareto front in less than 0.1 seconds, where the ground truth is obtained using a greedy algorithm and by sampling 20,000 points in the feasible set $K$. \Cref{fig:case1}a displays the geometric shape of $K$. We observe that all discovered Pareto optimal solutions are within $K$ (in particular, they lie on the boundary of $K$) and hence are feasible solutions of the constrained MOO problem.

\subsubsection{Implementation details}\label{sec:ex1_implementation}
To recover the Pareto front, we apply \Cref{alg:PD-KKT-constr} with parameters $\alpha =1, c = 0.1, \mu = 0.01$. We fix $x$ to be $(0,0)$ and vary $\tau $ from $(-10,10)$ to $(10,-10)$ along the diagonal path. We use at most $\texttt{maxit\_outer} = 100$ iterations and a stopping tolerance of $\varepsilon = 10^{-5}$.
Since $\Pi_K$ can be efficiently computed and the problem is sufficiently smooth, we update $u^{(j+1)}$ according to Remark~\ref{rem:u-update}. Specifically, 
we solve the $u$-subproblem~\eqref{eq:constrained_innermin} using projected gradient descent preconditioned by~\eqref{eq:constrained_precond} and with adaptive
  step size (backtracking with factor $\beta = 0.5$ and Armijo parameter  $c_1 = 10^{-4}$), at most $\texttt{maxit\_u} = 200$ iterations, and
  stopping tolerance $\texttt{tol\_u} = 10^{-4}$ on the iterate displacement. 
  The projection $\Pi_K$ 
  is implemented via Dykstra’s algorithm with a fixed number of $\texttt{proj\_cycles} = 10$ Dykstra cycles per projection, and scalar root-finding for the epigraph projection (cubic equation) is carried out with tolerance $\texttt{root\_tol} = 10^{-6}$.  
  We refer readers to~\cite{kelley1995iterative,nocedal2006numerical} for more details on the above fixed-point and backtracking strategies. 

\subsection{Example 2. Nonconvex Pareto fronts}\label{sec:ex_nonconvex}

We now consider two 2D MOO problems on a simple box, where  both are designed to exhibit \emph{nonconvex} Pareto fronts in the objective space. 
In both cases, the decision variable is $u = (u_1,u_2) \in K = [0,1]^2$. We compare against 
both the ground-truth Pareto front obtained via greedy Pareto selection on an exhaustive sampling of $K$ on a grid as well as the convex envelope of the true Pareto front in the objective space.
Specifically, to compute the ground-truth Pareto front, we evaluate $\ell$ on a $150 \times 150$ uniform grid in $K = [0,1]^2$ and obtain the Pareto points via a greedy Pareto selection (minimization). To compute the convex envelope, we minimize various linear weighted sums of the two objectives.
The implementation details for our HJ/Hopf-Lax solver are the same as in~\Cref{sec:ex1_implementation}. 
\begin{figure}
    \centering
    \subfigure[Case 1. Nonconvex example.]{
      \includegraphics[width=0.4\linewidth]{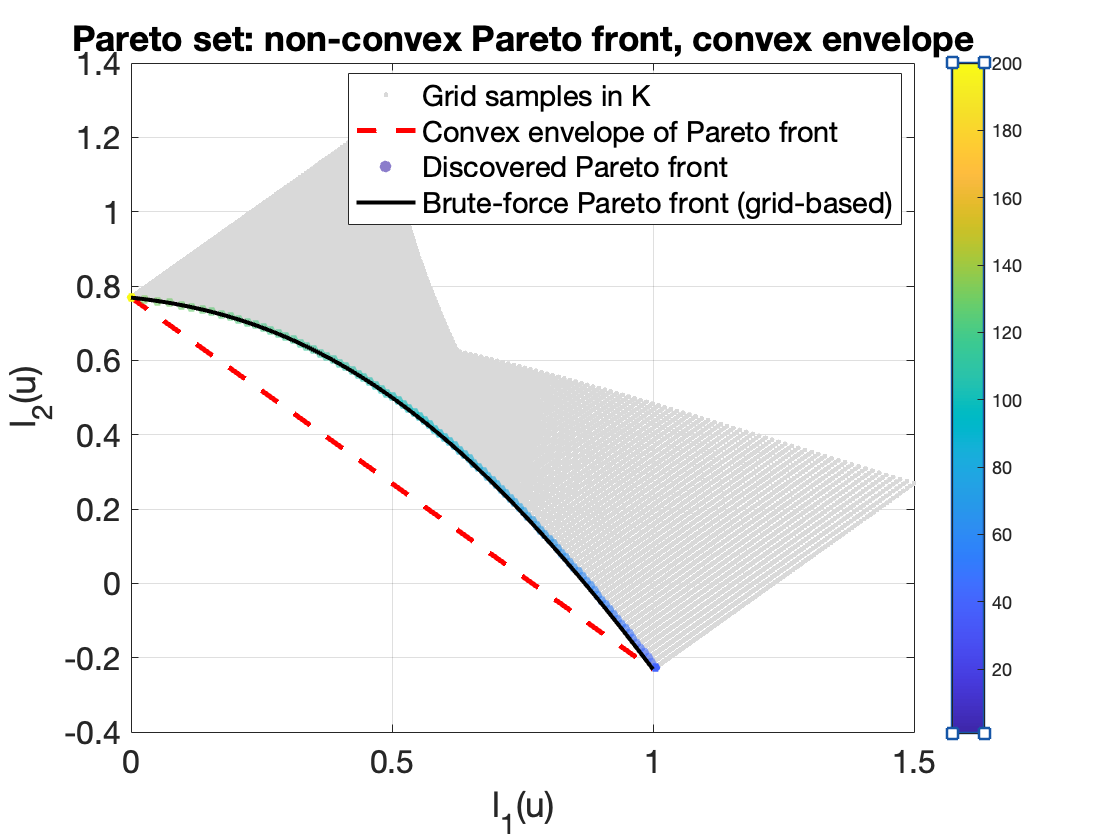}}
    \subfigure[Case 2. Highly nonconvex example.]{
      \includegraphics[width=0.4\linewidth]{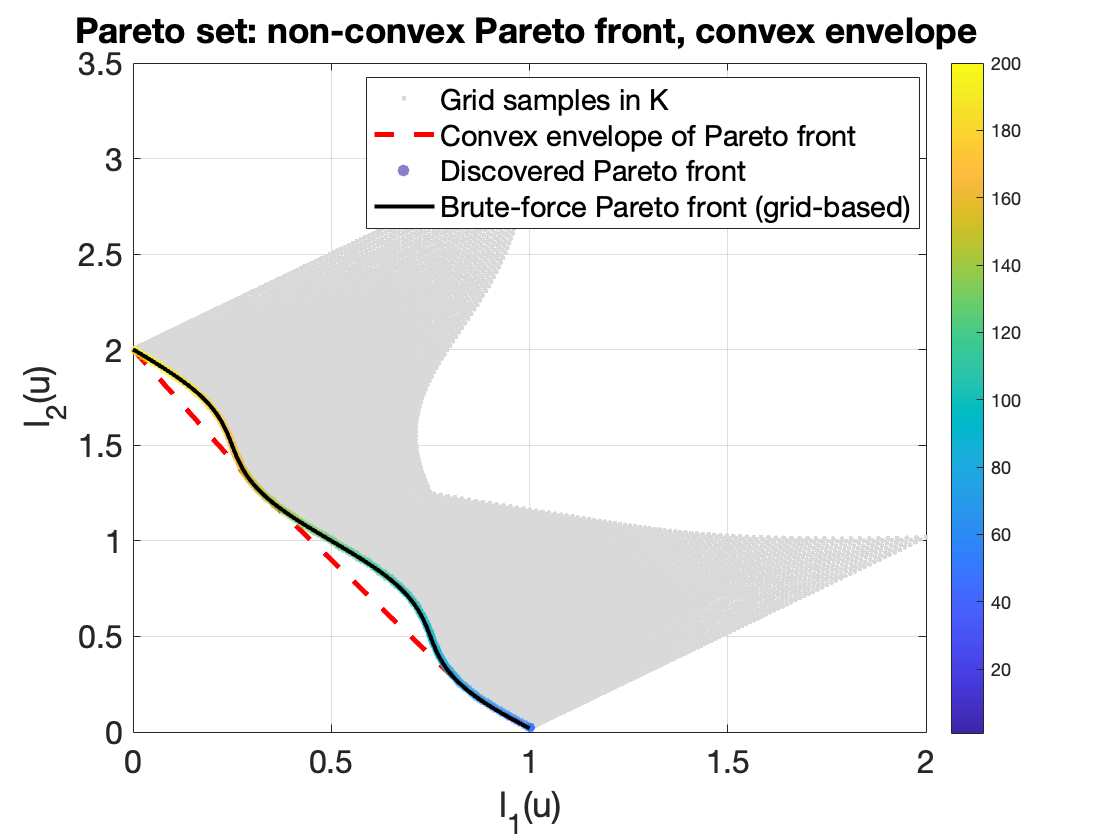}}
    \caption{\textit{Nonconvex Pareto fronts with $K= [0,1]^2$.} Whereas most conventional MOO algorithms only recover the convex envelopes (\textcolor{red}{$--$}) of nonconvex Pareto fronts (\textbf{--}), our primal-dual HJ/Hopf-Lax algorithm (colored dots) is able to perfectly discover highly nonconvex Pareto fronts, including regions that lie strictly above their convex envelopes. In both cases, the computational time of our algorithm is less than 0.1s.}
    \label{fig:case2}
\end{figure}

We consider the following two cases:
\begin{enumerate}
    \item  Define the two objectives as
\begin{equation}
\left\{
\begin{aligned}
    & \ell^1_1(u) = u_1 + \lambda (u_2 - u_1)^2, \\  
    &\ell^1_2(u) = 1 - u_1 
      + a (u_1 - 0.5)^4 
      - b (u_1 - 0.5)^2 
      + \lambda (u_2 - u_1)^2,
      \end{aligned}
      \right.
\end{equation}
where $a = 0.3$, $b = 1$, and $\lambda = 0.5$. For each fixed $u_1$, both $\ell^1_1$ and $\ell^1_2$ are minimized at $u_2 = u_1$, so the Pareto set in the decision space should be close to the diagonal (i.e., $u_2 \approx u_1$ in $K$), and, due to the structure of $\ell^1_2$, its  image $(\ell^1_1,\ell^1_2)$ forms a smooth but nonconvex curve in the objective space.   

 \item Define the two objectives as
\begin{equation}
\left\{
\begin{aligned}
    & \ell^2_1(u) =  u_1 + \gamma_1 \sin(4\pi u_1) + \beta_1 (u_2 - u_1)^2, \\  
    &\ell^2_2(u) = (u_1 - 0.25)^4 (u_1 - 0.75)^2 + \eta (1 - u_1) + \beta_2 (u_2 - u_1)^2 ,
      \end{aligned}
      \right.
\end{equation}
where $\gamma_1 = 0.05$, $\beta_1 = \beta_2 = 1$, and $\eta = 2$. 
The variance term $(u_2 - u_1)^2$ in both objectives penalizes deviations from the diagonal $u_1 = u_2$, so Pareto optimal solutions are again expected to cluster near $u_2 \approx u_1$.
On the diagonal, one has $(u_2 - u_1)^2 = 0$, and the induced curve $t \mapsto (\ell_1^2(t,t),\ell_2^2(t,t))$ forms a highly nonconvex front in the objective space. 
\end{enumerate}

In \Cref{fig:case2}, we see that the convex envelope (\textcolor{red}{$--$}) lies below the true Pareto front (\textbf{--}), as expected since the Pareto front is nonconvex. Most conventional MOO algorithms generally are only able to recover the convex envelope of nonconvex Pareto fronts.
In contrast, our algorithm (colored dots) accurately resolves the whole Pareto front, including nonconvex portions that lie strictly above the convex envelope. In both cases, the computational time for our algorithm is less than 0.1 seconds.



\subsection{Example 3. High-dimensional, nonconvex MOO}

Here, we consider two high-dimensional problems with decision variables $u \in [0,1]^d$ for various $d$. Define
\begin{equation}\label{highd_sr}
    s(u) \coloneqq \frac{1}{d} \sum_{i=1}^d u_i,
    \qquad
    r^2(u) \coloneqq \frac{1}{d} \sum_{i=1}^d (u_i - s(u))^2,
\end{equation}
which can be thought of as an analogue of the mean and variance of the decision variables, respectively, and will be used to define the objective functions. 

\subsubsection{Case 1. High-dimensional decision space}\label{sec:ex3_case1}
We consider the following objective functions:
\begin{equation}
    \left\{
    \begin{aligned}
    \ell_1(u)
    &= s(u) + \gamma_1 \sin\bigl(2\pi s(u)\bigr) + \beta_1\, r^2(u) \ , \\
    \ell_2(u)
    &= 1 - s(u)
       + a \bigl(s(u)-\tfrac{1}{2}\bigr)^4
       - b \bigl(s(u)-\tfrac{1}{2}\bigr)^2
       + \beta_2\, r^2(u),
\end{aligned}
    \right.
\end{equation}
where $a = 1, b = 0.7, 
    \gamma_1 = 0.1,
    \beta_1 = \beta_2 = 0.5.$ Similarly to~\Cref{sec:ex_nonconvex}, along the diagonal $u = t \mathbf{1}_d$ (where $\mathbf{1}_d$ denotes a vector of 1s), one has $r^2(u) = 0$ and the problem reduces to a 1D smooth but nonconvex optimization problem in the scalar variable $s(u)=t$. 
Away from the diagonal, the variance term $r^2(u)$ penalizes anisotropic configurations, so that the Pareto set is constrained to lie near the diagonal but in a high-dimensional ambient space.

\begin{figure}
    \centering
    \subfigure[$d = 3$.]{
      \includegraphics[width=0.33\linewidth]{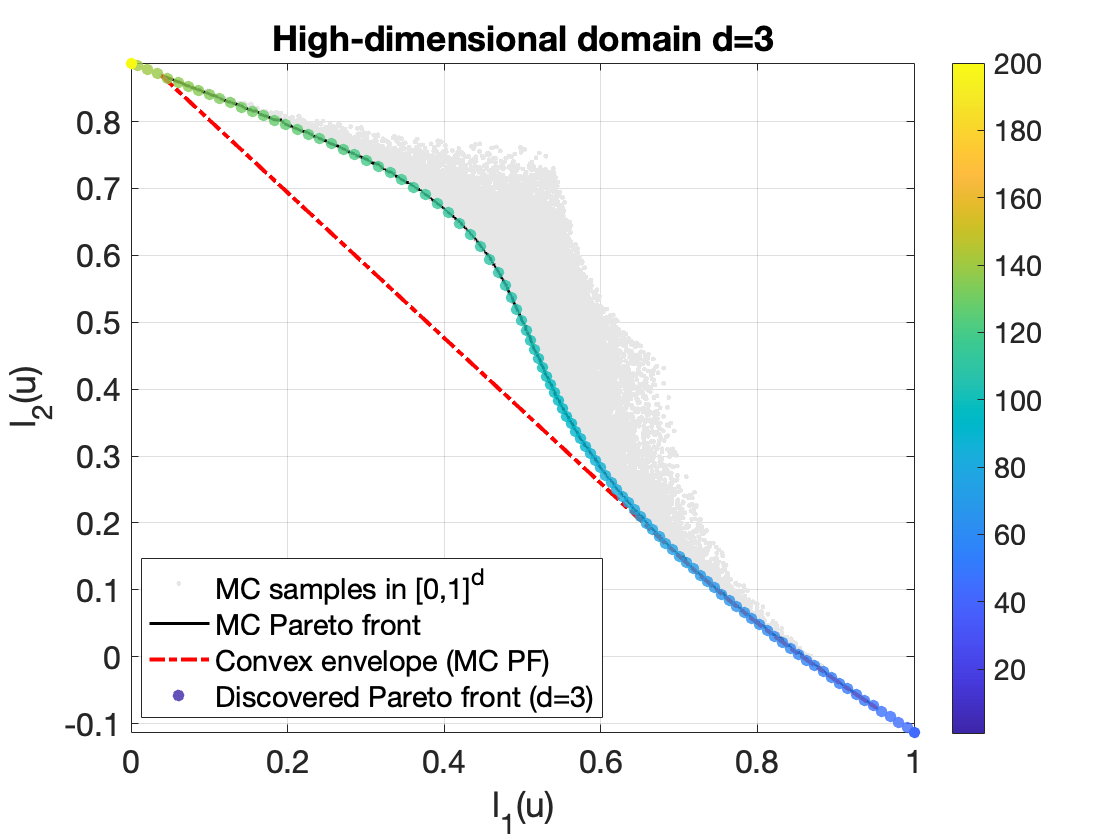}}
    \subfigure[$d = 10$.]{
      \includegraphics[width=0.33\linewidth]{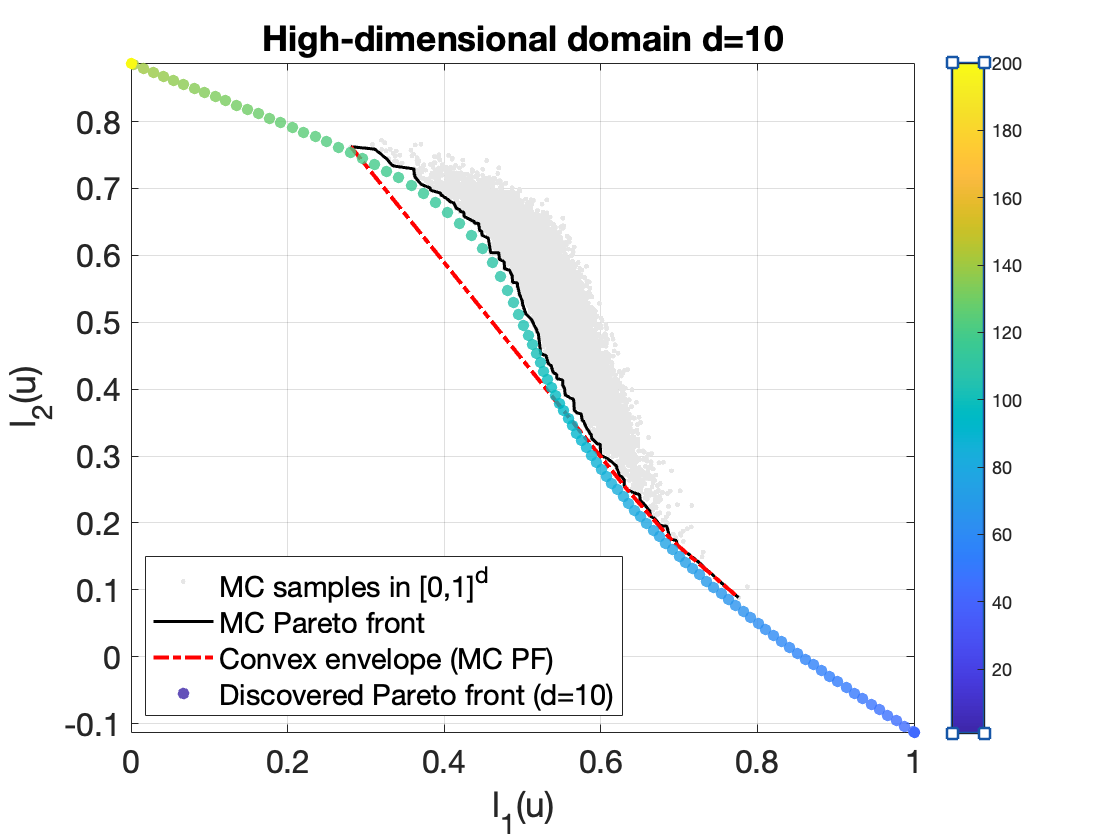}}
      \subfigure[$d = 30$.]{
      \includegraphics[width=0.33\linewidth]{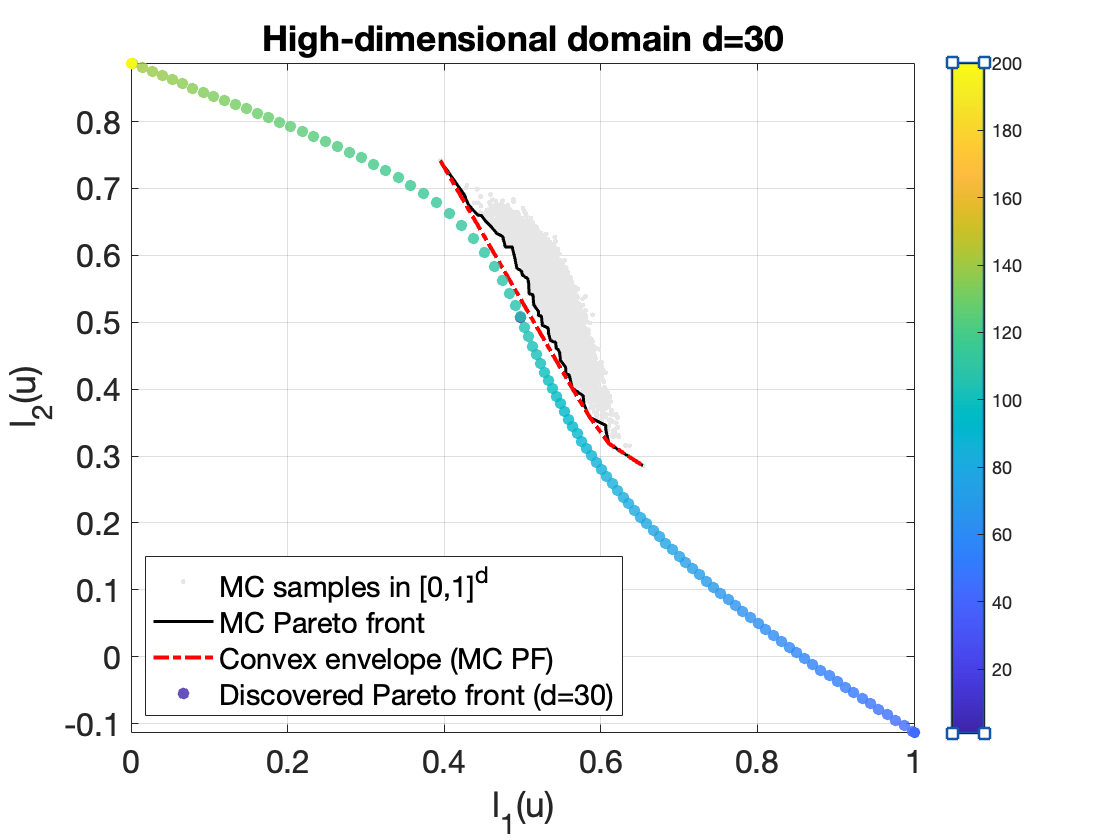}}
      \subfigure[$d = 100$.]{
      \includegraphics[width=0.33\linewidth]{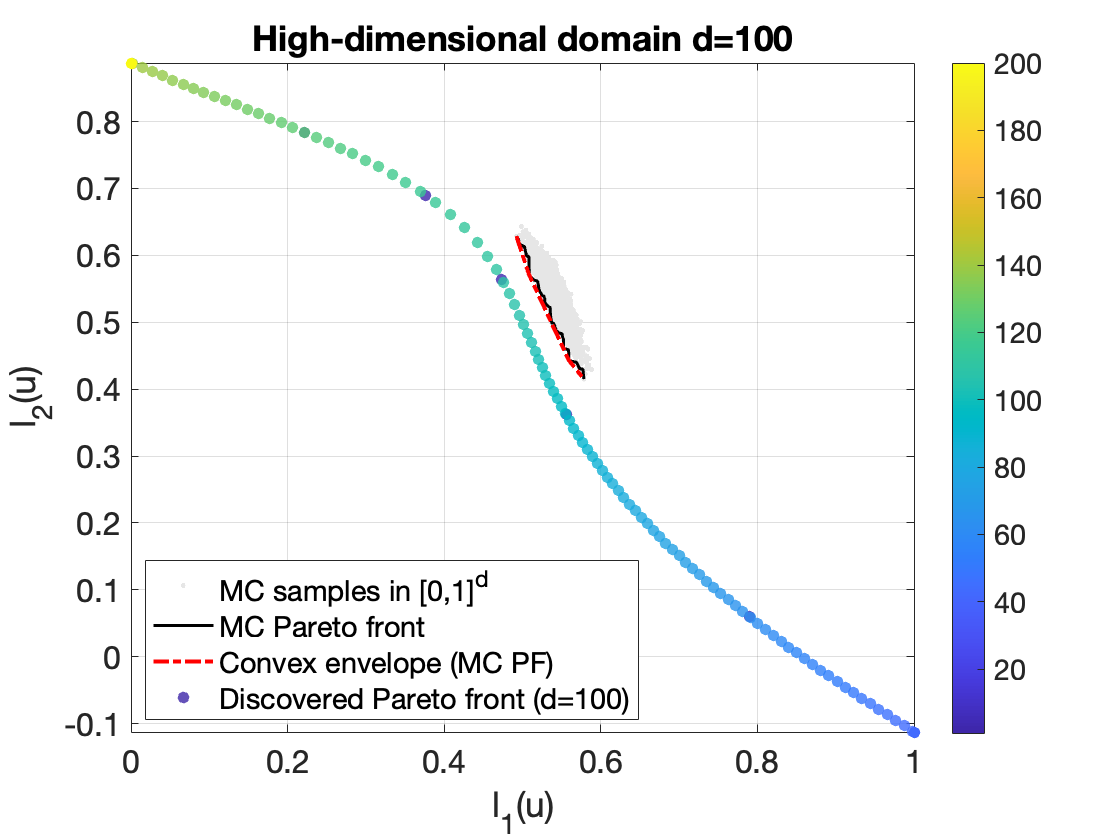}}
    \caption{\textit{Nonconvex Pareto fronts with high-dimensional decision spaces.} Our HJ/Hopf-Lax solver (colored dots) is able to discover nonconvex Pareto fronts with various high-dimensional decision domains $K = [0,1]^d$, including portions that lie strictly above their convex envelope (\textcolor{red}{$--$}). In comparison, we are able to more extensively and continuously explore the Pareto front than the brute force method (greedy Pareto selection with iid Monte Carlo sampling), which, in contrast, only yields an \textit{approximate}, \textit{discrete} reference (\textbf{--}) and becomes exponentially more intractable as $d$ increases. }
    \label{fig:case3}
\end{figure}

We implement our HJ/Hopf-Lax solver similarly to the previous examples but
 scale the parameters as $\varepsilon = 0.1, 
    \alpha = 1, 
    c = \frac{0.1}{d},
    \mu = \frac{0.01}{d}$
    in order to maintain stability as $d$ grows. Analogously to before, we fix $x = 0$ and vary $\tau$ along the diagonal path from $-10\mathbf{1}_d$ to $10\mathbf{1}_d$.
All other implementation details remain the same as in~\Cref{example1}. 
    Since a full tensor grid in $[0,1]^{100}$ is computational intractable, we approximate a reference Pareto front using Monte Carlo sampling. Specifically, we draw $N_{\mathrm{mc}} = 2\times 10^4$ i.i.d. decision vectors $u^{(j)} \sim \mathcal{U}([0,1]^d)$, evaluate $\ell(u^{(j)})$ for each sample, and apply a greedy Pareto selection (for minimization) in the objective space to extract an approximate Monte Carlo Pareto front. Its convex envelope is again computed by minimizing various linear combinations of the objectives. 
    
    In~\Cref{fig:case3}, we see that our algorithm matches the general shape of the approximate nonconvex reference Pareto front for various dimensions $d = 3,10,30,100$, even in regions that lie strictly above their convex envelope. We are also able to more extensively explore the Pareto front than Monte Carlo sampling, which highlights the improved tractability of our algorithm in high dimensions over more traditional approaches. In~\Cref{tab:highd-times}, we observe that the computational runtimes of our algorithm scale polynomially in $d$, thereby mitigating the curse of dimensionality for MOO. 

\begin{table}[t]
    \centering
    \begin{tabular}{c c}
        \hline
        Dimension $d$ & Computational time (s) \\
        \hline
        $3$   & $< 0.1$ \\
        $10$  & $0.25$ \\
        $30$  & $16.73$ \\
        $50$  & $35.56$ \\
        $100$ & $100.37$ \\
        \hline
    \end{tabular}
    \caption{\textit{Computational runtime of our HJ/Hopf-Lax solver for various decision space dimensions $d$.} The runtime of our algorithm scales polynomially with $d$, which highlights its capability to mitigate the curse of dimensionality for MOO.  All timing results are obtained using an Apple M3 Pro CPU.}
    \label{tab:highd-times}
\end{table}

\subsubsection{Case 2. High-dimensional decision and objective spaces}\label{sec:real-world-inspired-example}

We consider a high-dimensional example in which the dimensions of both the decision and objective spaces are large. Let $d=20$, and define $s(u)$ and $r(u)$ by~\eqref{highd_sr}. We consider the following $N_{\mathrm{obj}} = 5$ objectives:
\begin{equation}\label{eq:highd-multi-ell}
\left\{
\begin{aligned}
    \ell_1(u)
    &= s(u) + \gamma_1 \sin\bigl(2\pi s(u)\bigr)
            + \beta_1\, r^2(u), \\
    \ell_2(u)
    &= 1 - s(u)
       + a\bigl(s(u) - \tfrac{1}{2}\bigr)^4
       - b\bigl(s(u) - \tfrac{1}{2}\bigr)^2
       + \beta_2\, r^2(u), \\
    \ell_3(u)
    &= \bigl(s(u) - 0.2\bigr)^2 + c_3\, r^2(u), \\
    \ell_4(u)
    &= \bigl(s(u) - 0.8\bigr)^2 + c_4\, r^2(u), \\
    \ell_5(u)
    &= \tfrac{1}{2}\, s(u)^2 + \gamma_5 \sin\bigl(4\pi s(u)\bigr)
       + c_5\, r^2(u),
\end{aligned}
\right.
\end{equation}
where $
    a = 1.0,  b = 0.7, 
    \gamma_1 = 0.1, 
    \beta_1 = \beta_2 = 0.5, 
    c_3 = 0.3, \; c_4 = 0.4, \; c_5 = 0.2, 
    \gamma_5 = 0.05.
$  We implement our algorithm and obtain approximate reference Pareto fronts and their convex envelopes identically to~\Cref{sec:ex3_case1}.

While this problem is not directly derived from a specific
physical model, its structure is inspired by typical
quantities arising in many-body and multi-agent systems. 
The scalar
$s(u)$ can be interpreted as an average state or control level
(e.g., mean load, temperature, or consensus variable), while the
variance term $r^2(u)$ can be viewed as penalizing spatial heterogeneity or disagreement
across agents. 
The multi-well and oscillatory contributions in each of the objectives $\ell_i$ mimic energy landscapes with multiple preferred
operating regimes and resonance-like effects. 
As such, this example remains representative of realistic trade-offs between global performance and dispersion that arise in many real-life applications.

 In~\Cref{fig:case3_2}, we show projections of our recovered Pareto front (colored dots) into various 2D objective subspaces. As before, we see that we generally match the shape of the approximate reference Pareto front (\textbf{--}), discovering nonconvex regions that lie strictly above its convex envelope (\textcolor{red}{$--$}). Note that while our algorithm recovers continuous 1D curves along the Pareto front, some of the plots appear discontinuous since they only show 2D projections. The runtime of our algorithm is 13.92s, which demonstrates its potential for real-time, high-dimensional, real-world applications.

\begin{figure}
    \centering
    \includegraphics[width=0.75\linewidth]{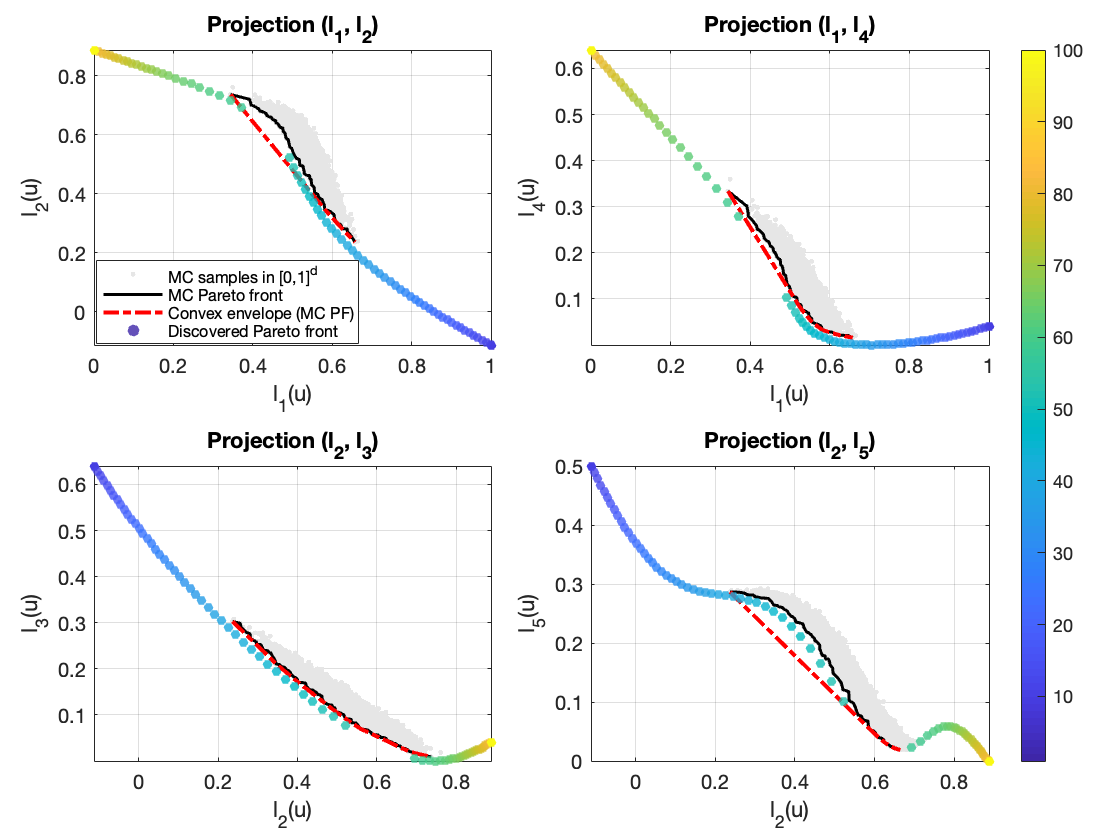}
    \caption{\textit{Projections of a nonconvex Pareto front with high-dimensional decision and objective space into 2D objective space.} Our HJ/Hopf-Lax solver is able to recover a nonconvex Pareto front with 20D decision space and 5D objective space in just 13.92s, which highlights its efficiency and tractability even in very high dimensions. While not derived from a particular physical model, the structure of this MOO problem is inspired by those arising in many-body, multi-agent systems, which demonstrates the potential of our algorithm for high-dimensional, real-world problems. }
    \label{fig:case3_2}
\end{figure}

\section{Summary}\label{summary}
Starting from a monotone preference function, we embedded the Pareto optima for MOO into a parameterized differential game whose upper value solves a first-order HJ equation and admits a Hopf–Lax representation. 
Under mild regularity assumptions, the resulting family of representations traces a dense subset of an induced portion of the weak Pareto front. 
In particular, by allowing for more general convex, monotone preference functions, this representation enables the recovery of Pareto fronts with arbitrarily nonconvex regions that cannot be captured by conventional convex-envelope approximation methods, such as weighted-sum scalarizations. Notably, this result requires very few assumptions on the MOO problem (the objectives just need to be lsc and proper). We then extended this representation to constrained MOO problems that satisfy standard constraint qualifications.

In both cases, we leveraged our new representations to develop efficient primal–dual algorithms for the resulting optimality systems that scale polynomially in the dimension of the decision and objective spaces, thereby mitigating the curse of dimensionality.  
We then numerically demonstrated that our algorithms can efficiently capture nonconvex Pareto geometries in very high dimensions, exposing continuous 1D curves along these fronts. While we use the soft-max function in our experiments, any other preference function that satisfies~\Cref{assump_g} could instead be deployed. In general, the main factors in the choice of preference function is the computational efficiency of its implementation and its domain (which affects how much of the Pareto front can be exposed). Moreover, recovering full coverage of higher dimensional surfaces still poses some challenges, regardless of this choice. 

Our new approach helps bridge the gap between classical single-objective optimization tools (e.g., primal–dual methods, continuation/homotopy, proximal splitting, constraint-handling techniques) and MOO. One future direction is to extend this framework 
to broader multi-objective settings, including richer objective classes, more general feasible sets, and state constraints. 
Our primal-dual algorithm already shows promise for real-world problems. 
One possible application would be to efficiently expose high-dimensional, nonconvex Pareto fronts arising in otherwise expensive engineering and data-driven workflows. For instance, applying our method to MOO problems in machine learning (e.g., accuracy–robustness tradeoffs, multi-task/multi-loss training, learning with safety constraints/regularization) would allow for interpretable Pareto front exploration beyond convex envelope approximations. 

Another interesting direction would be to extend our approach to multi-objective optimal control (MOOC)~\cite{de2009class}, which is a form of infinite-dimensional MOO that replaces the static vector-valued objective with a functional that evolves along a trajectory governed by a controlled dynamical system. In MOOC, one typically seeks controls that steer the state, while jointly optimizing several integral and/or terminal costs, so that Pareto optimality is defined over the space of admissible controls and the associated cost vectors. 
While MOOC problems have also been studied within the HJ framework (see, e.g., \cite{guigue2013set,kumar2010efficient,takei2015optimal,desilles2019pareto}) and
several MOOC solvers based on the combination of discretization, scalarization, and optimization have been proposed (see, e.g., \cite{bellaassali2004necessary, bonnel2010optimization, de2016sufficient}), extending the framework developed here may provide new interpretations and new potential alternative algorithms for this important field.


\bibliographystyle{siamplain}
\bibliography{references}

\end{document}